\documentclass[12pt]{amsart}

\usepackage{amsfonts,amssymb,amsmath,amsxtra,mathtools,bm}
\usepackage{txfonts}
\usepackage{graphicx}
\usepackage{booktabs}
\usepackage{placeins}
\usepackage{cite}
\usepackage[shortlabels]{enumitem}
\usepackage{microtype}
\usepackage{tikz}
\usetikzlibrary{arrows.meta,calc,angles,quotes}

\usepackage[hidelinks]{hyperref}

\allowdisplaybreaks
\numberwithin{equation}{section}

\makeatletter
\@ifundefined{KV@Gin@alt}{\define@key{Gin}{alt}{}}{}
\makeatother

\newcommand{\tikzalt}[1]{}
\newcommand{\tikzaltend}{}

\newcommand{\dd}{\,\mathrm{d}}
\newcommand{\R}{\mathbb{R}}
\newcommand{\cT}{\mathcal{T}}
\newcommand{\cE}{\mathcal{E}}
\newcommand{\cV}{\mathcal{V}}

\theoremstyle{plain}
\newtheorem{theorem}{Theorem}[section]
\newtheorem{proposition}[theorem]{Proposition}
\newtheorem{lemma}[theorem]{Lemma}

\theoremstyle{definition}

\theoremstyle{remark}
\newtheorem{remark}[theorem]{Remark}

\newcommand{\papertitle}{A Finite Element Method for the Surface Biharmonic Problem\texorpdfstring{\\}{ }%
without Stabilization or Piola Transforms}

\hypersetup{
  pdftitle={\papertitle},
  pdfauthor={JiaChen Gu, Xia Ji, Shuo Zhang},
  pdfsubject={Nonconforming finite element methods for surface biharmonic problems},
  pdfkeywords={surface biharmonic problem, nonconforming finite element method, reduced enriched quartic element, discrete Miranda--Talenti-type estimate, piecewise planar surface}
}

\begin{document}

\title[Surface Biharmonic Elements without Stabilization or Piola Transforms]{\papertitle}

\author{JiaChen Gu}
\address{School of Mathematics and Statistics, Beijing Institute of Technology,
Beijing 100081, China}
\email{jiachen\_gu@bit.edu.cn}

\author{Xia Ji}
\address{School of Mathematics and Statistics, Beijing Institute of Technology,
Beijing 100081, China}
\address{Beijing Key Laboratory on MCAACI, Beijing Institute of Technology,
Beijing 100081, China}
\address{MSU-BIT-SMBU Joint Research Center of Applied Mathematics, Shenzhen
MSU-BIT University, Shenzhen 518172, China}
\email{jixia@bit.edu.cn}

\author{Shuo Zhang}
\address{State Key Laboratory of Mathematical Sciences and State Key Laboratory
of Scientific and Engineering Computing, Institute of Computational Mathematics
and Scientific/Engineering Computing, Academy of Mathematics and Systems Science,
Chinese Academy of Sciences, Beijing 100190, China}
\address{School of Mathematical Sciences, University of Chinese Academy of
Sciences, Beijing 100049, China}
\email{szhang@lsec.cc.ac.cn}

\thanks{This work is partially supported by the Strategic Priority Research
Program of the Chinese Academy of Sciences (Grant No. XDB0640000; S. Z.), the
National Natural Science Foundation of China (Grant Nos. 12371389 and 12171036;
X. J., and Grant No. 12271512; S. Z.), the Beijing Natural Science Foundation
(Grant No. Z210001; X. J.), and the Fundamental Research Funds for the Central
Universities (X. J.).}

\begin{abstract}
For the biharmonic problem on closed smooth surfaces, we construct a nonconforming finite element method that requires neither stabilization nor the Piola transform. The reduced enriched quartic (REQ) element relies solely on vertex values and scalar edge moments. The same local construction applies to every triangle of a piecewise planar surface. The bilinear form consists of elementwise Laplacian inner products and permits standard finite element assembly. By separating the planar vertex residual from the angle defect contribution, we establish weak discrete Miranda--Talenti-type estimates on both planar meshes and closed piecewise planar surfaces. These estimates guarantee uniform stability on sufficiently refined interpolating surface meshes. We prove the well-posedness of the method, first-order convergence in the energy norm, and second-order convergence in the broken $H^1$ and $L^2$ norms. Numerical tests on three surfaces show the predicted convergence rates and examine the discrete Miranda--Talenti constants and angle defect scaling.
\end{abstract}



\maketitle
\tableofcontents

\section{Introduction}
\label{sec:introduction}
Higher-order surface equations arise in models of surface plate bending, membrane structures, and phase separation; representative examples include the surface Kirchhoff plate equation and the surface Cahn--Hilliard equation \cite{Walker2022,ElliottRanner2015}. In this paper, we study the surface biharmonic problem
\begin{equation}
	\Delta_\gamma^2u=f\quad\text{on }\gamma,
	\qquad
	\int_\gamma u\dd\sigma=0,
	\qquad
	\int_\gamma f\dd\sigma=0.
	\label{eq:intro-surface-biharmonic}
\end{equation}
Here, $\gamma\subset\mathbb R^3$ is a $C^4$ compact, connected, closed, and orientable two-dimensional surface, $\Delta_\gamma$ is the Laplace--Beltrami operator, and $\dd\sigma$ is the surface measure. Let $\mathring H^2(\gamma)=\{v\in H^2(\gamma):\int_\gamma v\dd\sigma=0\}$, and denote the $L^2(\gamma)$ inner product by $(\cdot,\cdot)_\gamma$. For $f\in L^2(\gamma)$ with zero mean, the variational form of \eqref{eq:intro-surface-biharmonic} reads: find $u\in\mathring H^2(\gamma)$ such that
\begin{equation}
	a(u,v):=(\Delta_\gamma u,\Delta_\gamma v)_\gamma
	=(f,v)_\gamma=:l(v)
	\qquad\forall v\in\mathring H^2(\gamma).
	\label{eq:continuous-problem-new}
\end{equation}

Following Dziuk's surface finite element framework \cite{Dziuk1988}, we approximate $\gamma$ by a piecewise planar triangulation $\Gamma_h$ with faces $K\in\mathcal T_h$; see \cite{DziukElliott2013,BonitoDemlowNochetto2020} for the geometric and approximation theories. Our aim is to construct a finite element method with scalar assembly, without Piola transforms or stabilization, based on the bilinear form
\begin{equation}
	a_h(w_h,v_h):=\sum_{K\in\mathcal T_h}\int_K
	\Delta_Kw_h\,\Delta_Kv_h\dd\sigma_h,
	\label{eq:intro-discrete-laplacian-form}
\end{equation}
where $\Delta_K$ is the Laplacian on the plane containing $K$, and $\dd\sigma_h$ is its area measure. Retaining the Laplacian form is relevant on curved surfaces: the Hessian energy differs from it by a Gaussian curvature term, as expressed by the Bochner--Grisvard identity
\begin{equation}
	(\Delta_\gamma u,\Delta_\gamma v)_\gamma
	=(\nabla_\gamma^2u,\nabla_\gamma^2v)_\gamma
	+(K_\gamma\nabla_\gamma u,\nabla_\gamma v)_\gamma,
	\qquad u,v\in H^2(\gamma),
	\label{eq:surface-bochner-grisvard}
\end{equation}
where $\nabla_\gamma^2$ is the covariant Hessian and $K_\gamma$ is the Gaussian curvature; see \cite{Reusken2020}, \cite[Lemma~2.7]{NeilanWan2025}, and \cite[Remark~2.1]{WuZhou2025}. The form \eqref{eq:intro-discrete-laplacian-form} directly discretizes \eqref{eq:continuous-problem-new} without an explicit curvature coefficient. Its stability requires control of the full broken Hessian by the piecewise Laplacian.

Existing methods for fourth-order surface equations approach these issues in several ways. Second-order splitting and mixed formulations introduce auxiliary unknowns and use continuous lower-order elements \cite{DziukElliott2008Eulerian,ElliottFritzHobbs2019,SteinGrinspunJacobsonWardetzky2019}. Walker \cite{Walker2022} developed a surface Hellan--Herrmann--Johnson method for Kirchhoff plates, including curved elements and various boundary conditions, while Pan--Rabczuk--Xu--Chen \cite{PanRabczukXuChen2019} combined extended Loop subdivision with mixed formulations. Primal discretizations include the continuous/discontinuous Galerkin method of Larsson--Larson \cite{LarssonLarson2017} for the surface biharmonic equation and the $C^0$ interior penalty method of Neilan--Wan \cite{NeilanWan2025} for surface Stokes stream functions. The latter avoids explicit Gaussian curvature calculations. On unfitted meshes, the TraceFEM of Neilan--Wan \cite{NeilanWanTrace2026} combines surface edge penalties with bulk-facet stabilization. Cai--Guo--Zhang \cite{CaiGuoZhang2026} use continuous linear elements with gradient recovery and stabilization.

For a primal method based on \eqref{eq:intro-discrete-laplacian-form}, a natural starting point is the planar theory of $H^2$ finite elements and discrete Miranda--Talenti estimates. Many classical planar $H^2$ elements use vertex derivatives as degrees of freedom to impose derivative continuity. Vertex derivative continuity also features in the planar Miranda--Talenti constructions reviewed by Gallistl--Tian \cite{GallistlTian2025}. Their two-dimensional continuous, $H^2$-nonconforming elements are $C^1$ at vertices and satisfy a strong discrete Miranda--Talenti property: the $L^2$ norm of the broken Hessian is bounded by that of the piecewise Laplacian with constant one. These planar constructions make elements with vertex derivative degrees of freedom natural candidates for extension to surfaces.

On a piecewise planar surface mesh, adjacent faces may be noncoplanar, and their vertex gradients belong to different tangent planes. This precludes the construction of $H^2$-conforming finite element spaces on folded surfaces and makes the extension of planar elements with vertex derivative continuity nontrivial. Inter-element Piola transforms provide a way to relate these tangential vectors, as in the surface Stokes method of Demlow--Neilan \cite{DemlowNeilan2024}. Wu--Zhou \cite{WuZhou2025} used this construction to extend the planar NZT element of Wang--Shi--Xu \cite{WangShiXu2007} to surfaces. Piola coupling allows the vertex derivative degrees of freedom to be used on the folded mesh, but does not itself provide the Miranda--Talenti control required by the pure Laplacian form. Their method augments the Laplacian energy with conormal derivative jump stabilization. For the surface NZT method, Wu--Zhou proved well-posedness, first-order convergence in the energy norm, and second-order convergence in the broken $H^1$ norm. Stabilization also plays a role in planar discrete Miranda--Talenti analysis, as illustrated by the jump estimates of Neilan--Wu \cite{NeilanWu2019} and the stabilized Hermite methods of Wu \cite{Wu2021}.

Based on conormal derivative jump stabilization, Wu--Zhou \cite{WuZhouMorley2026} further developed a stabilized Morley method for the stream-function formulation of surface Stokes. Its scalar degrees of freedom allow direct assembly on the piecewise planar surface without inter-element Piola transforms. The method combines trace-free Hessian and tangential gradient inner products with value jump stabilization, and achieves first-order convergence in the broken $H^2$ norm and second-order convergence in the broken $H^1$ norm. This energy differs from \eqref{eq:intro-discrete-laplacian-form}; already in the plane, the standard Morley space does not satisfy a discrete estimate controlling the broken Hessian by the piecewise Laplacian \cite{GallistlTian2025}.

In this paper, we construct a reduced enriched quartic (REQ) element that combines scalar assembly with the unstabilized Laplacian form \eqref{eq:intro-discrete-laplacian-form} and without the aid of Piola transforms. Starting from the 18-dimensional space of Zhang \cite[section~2.3.2]{Zhang2018}, we impose two constraints on each edge: the first-order function moment is determined by the endpoint values, and the second-order conormal derivative moment vanishes. The resulting 12-dimensional space has three vertex values and, on each edge, one function moment and two conormal derivative moments as degrees of freedom. It contains all quadratic polynomials and admits optimal-order local interpolation. Conormal derivatives are evaluated within each element plane, so the same scalar assembly rules apply to planar and folded meshes. We prove a weak discrete Miranda--Talenti-type estimate, meaning uniform equivalence of the Laplacian and Hessian energies with constants not necessarily equal to one.

The edge constraints and assembly conditions are central to this estimate. Elementwise integration by parts reduces the difference between the Laplacian and Hessian energies to vertex residuals. The REQ space does not impose vertex gradient continuity, and these residuals need not vanish. Instead, we control them on vertex patches: after unfolding a folded patch, the residual separates into a planar contribution and a correction governed by the angle defect. A characterization of the local harmonic kernel gives a bound for the planar contribution in terms of the patchwise Laplacian energy; the angle defect correction is bounded by a lower-order gradient term. This decomposition yields the weak Miranda--Talenti estimate on planar meshes with homogeneous boundary conditions and, under the stated geometric assumptions, on sufficiently refined closed surface meshes.

On every connected closed triangulation with non-degenerate triangles, the global harmonic kernel of the REQ space consists solely of constants. Thus the Laplacian form is positive definite on the zero-mean space. For the shape-regular, quasi-uniform interpolating surface mesh families considered here, the discrete Miranda--Talenti estimate further gives uniform control of the full broken $H^2$ norm on that space for sufficiently small $h$. Under the stated regularity assumptions, approximation and geometric consistency estimates yield first-order convergence in the energy norm and second-order convergence in the broken $H^1$ and $L^2$ norms.

Numerical experiments on a sphere, a torus, and an implicit surface illustrate the predicted convergence rates. Additional experiments examine the discrete Miranda--Talenti constants and angle defect scaling on the same mesh families and their perturbations.

The remaining sections of the paper are organized as follows. Section~\ref{sec:b4pm-space-mt} develops the REQ element and proves the discrete Miranda--Talenti-type estimates on planar and folded meshes. Section~\ref{sec:error-analysis-new} specifies the continuous setting and surface geometry, formulates the discrete problem, and establishes error estimates. Sections~\ref{sec:numerical-experiments} and \ref{sec:conclusion} present numerical results and conclusions.

\section{The REQ Element and Discrete Miranda--Talenti-Type Estimates}
\label{sec:b4pm-space-mt}

The notation $A\lesssim B$ indicates $A\leq CB$, where $C$ depends solely on the mesh regularity, an upper bound on the vertex valence, and the fixed local finite element space, independent of the mesh size, curvature, and specific mesh entities.

\subsection{Piecewise planar triangulations and broken differential operators}
\label{sec:preliminaries}

Let $\mathcal M_h$ be an oriented, connected, shape-regular, and quasi-uniform piecewise planar triangulation obtained by gluing Euclidean triangles along common edges; the boundary $\partial\mathcal M_h\ne\varnothing$ is allowed. The sets of its triangles, edges, and vertices are denoted by $\cT_h(\mathcal M_h)$, $\cE_h(\mathcal M_h)$, and $\cV_h(\mathcal M_h)$, respectively. The subscripts $\mathrm{int}$ and $\partial$ denote interior and boundary entities, respectively. Let $h_K=\operatorname{diam}K$, $h_e=|e|$, and $h=\max_{K\in\cT_h(\mathcal M_h)}h_K$.

For an interior edge $e=K^+\cap K^-$ and a vertex $z$, let $\omega_e=K^+\cup K^-$, $\omega_z=\bigcup_{K\ni z}K$, and $h_z=\max_{K\subset\omega_z}h_K$. We assume uniformly bounded vertex valences; hence these patches have uniformly bounded complexity and finite overlap. Compatible orientations determine a unit normal vector $\nu_K$ on each $K\in\cT_h(\mathcal M_h)$.

The area measure on $\mathcal M_h$ and the arc length measure on its mesh edges are denoted by $\dd\sigma_h$ and $\dd s_h$, respectively. For any triangle or edge $D$, let $P_r(D)$ denote the space of polynomials of degree at most $r$ on $D$. For an integer $m\geq0$, define the broken spaces, norms, and seminorms as follows:
\begin{gather*}
	H_h^m(\mathcal M_h)
	:=\{v\in L^2(\mathcal M_h):v|_K\in H^m(K),
	\ \forall K\in\cT_h(\mathcal M_h)\},
	\\
	\|v\|_{m,h}^2
	:=\sum_{K\in\cT_h(\mathcal M_h)}\|v\|_{m,K}^2,
	\qquad
	|v|_{m,h}^2
	:=\sum_{K\in\cT_h(\mathcal M_h)}|v|_{m,K}^2.
\end{gather*}
For a piecewise function $v$, let
\[
(\nabla_hv)|_K=\nabla_K(v|_K),
\qquad
(D_h^2v)|_K=D_K^2(v|_K),
\qquad
(\Delta_hv)|_K=\Delta_K(v|_K),
\]
and write
\[
(v,w)_h=\sum_{K\in\cT_h(\mathcal M_h)}(v,w)_K,
\qquad
(v,w)_e=\int_e vw\dd s_h.
\]

For an interior vertex $z$, enumerate its adjacent triangles as $K_1,\ldots,K_m$, and let $\theta_{r,z}$ be the interior angle of $K_r$ at $z$. Following \cite{Christiansen2024Regge,Gawlik2020}, define the angle defect by:
\begin{equation}
	\kappa_z
	:=2\pi-\sum_{K_r\subset\omega_z}\theta_{r,z}.
	\label{eq:angle-defect-definition-new}
\end{equation}
The angle defect is signed, and no convexity assumption is imposed on the vertex patches. Successive rigid unfolding eliminates dihedral angles, and $\kappa_z$ is the closure defect after traversing a full turn around the vertex in the unfolded fan. In particular, $\kappa_z=0$ implies that the patch is intrinsically flat, but does not require adjacent triangles to be coplanar in $\mathbb R^3$. Figure~\ref{fig:vertex-star-unfolding} illustrates the angle defect as the closure defect of a fan obtained by cutting along a radial edge.

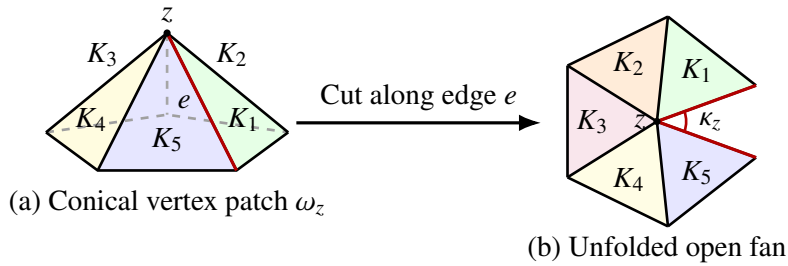
\begin{figure}[!htbp]
	\centering
	\tikzalt{Two-panel schematic of the angle defect at a vertex. Cutting the
		five-triangle conical patch around z along the radial edge e and rigidly
		unfolding it produces a planar open fan; the angular gap between the two
		copies of e is the angle defect at z.}
	\begin{tikzpicture}[scale=0.9,>=Latex,font=\small]
		\begin{scope}[xshift=-0.45cm]
			\coordinate (z)  at (0,1.12);
			\coordinate (r1) at (-1.78,-0.34);
			\coordinate (r2) at (0,-0.08);
			\coordinate (r3) at (1.78,-0.34);
			\coordinate (r4) at (1.02,-0.90);
			\coordinate (r5) at (-1.02,-0.90);
			\fill[purple!10] (z)--(r1)--(r2)--cycle;
			\fill[orange!12] (z)--(r2)--(r3)--cycle;
			\fill[green!10]  (z)--(r3)--(r4)--cycle;
			\fill[blue!10]   (z)--(r4)--(r5)--cycle;
			\fill[yellow!15] (z)--(r5)--(r1)--cycle;
			\draw[dashed,gray!75,line width=1pt] (r1)--(r2)--(r3);
			\draw[dashed,gray!70,line width=1pt] (z)--(r2);
			\draw[line width=1pt] (z)--(r1) (z)--(r3);
			\draw[line width=1pt] (r3)--(r4)--(r5)--(r1);
			\foreach \p in {r4,r5}{\draw[line width=1pt] (z)--(\p);}
			\draw[very thick,red!70!black] (z)--(r4);
			\node at (1.12,-0.15) {$K_1$};
			\node at (0.95,0.80) {$K_2$};
			\node at (-0.95,0.80) {$K_3$};
			\node at (-1.12,-0.15) {$K_4$};
			\node at (0,-0.45) {$K_5$};
			\node[left] at (0.51,0.11) {$e$};
			\fill (z) circle (1.5pt) node[above] {$z$};
			\node at (0,-1.38) {(a) Conical vertex patch $\omega_z$};
		\end{scope}
		
		\draw[->,very thick] (1.45,-0.20)--(5.05,-0.20)
		node[midway,above] {Cut along edge $e$};
		
		\begin{scope}[xshift=6.75cm,yshift=-0.18cm]
			\coordinate (o) at (0,0);
			\foreach \a/\b/\col in {
				20/84/green!10,84/148/orange!15,148/212/purple!10,
				212/276/yellow!12,276/340/blue!10}{
				\fill[\col] (o)--(\a:1.55)--(\b:1.55)--cycle;
				\draw[line width=1pt] (\a:1.55)--(\b:1.55);
				\draw[line width=1pt] (o)--(\a:1.55);
			}
			\draw[line width=1pt] (o)--(340:1.55);
			\draw[very thick,red!70!black] (o)--(20:1.55) (o)--(340:1.55);
			\node at (52:0.95) {$K_1$};
			\node at (116:0.95) {$K_2$};
			\node at (180:0.95) {$K_3$};
			\node at (244:0.95) {$K_4$};
			\node at (308:0.95) {$K_5$};
			\fill (o) circle (1.5pt) node[left] {$z$};
			\draw[line width=1pt,red!70!black]
			(340:0.43) arc[start angle=340,end angle=380,radius=0.43];
			\node[right] at (0.46,0) {$\kappa_z$};
			\node[align=center] at (0,-1.88) {(b) Unfolded open fan};
		\end{scope}
	\end{tikzpicture}
	\tikzaltend
	\caption{Angle defect at a vertex: the closure defect of the unfolded fan is $\kappa_z$.}
	\label{fig:vertex-star-unfolding}
\end{figure}

\subsection{Reduced enriched quartic (REQ) element}
\label{sec:local-req-element}

Let $K$ be an oriented non-degenerate planar triangle with positively oriented vertices $a_1,a_2,a_3$ and barycentric coordinates $\lambda_1,\lambda_2,\lambda_3$ (with $\lambda_4=\lambda_1$). Each edge inherits the counterclockwise orientation of $\partial K$; in the planar case, this orientation is the standard one on $\mathbb R^2$, whereas on a piecewise planar surface, its selection is compatible with $\nu_K$. Define $b_i=(\lambda_i^2\lambda_{i+1}-\lambda_i\lambda_{i+1}^2)\lambda_1\lambda_2\lambda_3$ and the enriched quartic shape function space
\begin{equation}
	S(K)=P_4(K)\oplus\operatorname{span}\{b_1,b_2,b_3\},
	\label{eq:enriched-quartic-space}
\end{equation}
where $S(K)\subset P_6(K)$. This space and its $18$ unisolvent degrees of freedom (three vertex values and edge moments defined below) are inherited from \cite[section~2.3.2 and Lemma~8]{Zhang2018}.

For an oriented edge $e$, denote its start and end vertices by $z_0$ and $z_1$, respectively, and parameterize it as $z_0+t(z_1-z_0)$ with $0\leq t\leq1$. Employ the orthogonal weights $m_0(t)=1$, $m_1(t)=\frac12-t$, and $m_2(t)=t^2-t+\frac16$. For $e\subset\partial K$, let $n_{K,e}$ be the unit outer conormal vector of $\partial K$ within the plane of $K$, and define the normalized moment degrees of freedom
\begin{align}
	D_e^r(v)
	&=\frac1{|e|}\int_e v\,m_r(t)\dd s_h,
	&&r=0,1,
	\label{eq:value-moments}\\
	D_e^{n,r,K}(v)
	&=\frac1{|e|}\int_e
	\partial_{n_{K,e}}v\,m_r(t)\dd s_h,
	&&r=0,1,2.
	\label{eq:normal-moments}
\end{align}
By imposing two linear reduction constraints on each edge of $K$, we define the local REQ space as
\begin{equation}
	V(K)
	=\left\{v\in S(K):
	\begin{aligned}
		&v(z_0)-v(z_1)-12D_e^1(v)=0,\\
		&D_e^{n,2,K}(v)=0
	\end{aligned}
	\quad\forall e\subset\partial K\right\}.
	\label{eq:filtered-local-space}
\end{equation}
These six linear constraints define the reduction from $S(K)$ to $V(K)$; reversing the edge orientation simultaneously reverses the sign of the endpoint difference and $m_1$, so \eqref{eq:filtered-local-space} is independent of the edge orientation.

\begin{proposition}[Local structure and interpolation]
	\label{lem:local-unisolvence-new}
	\label{lem:local-interpolation-new}
	The dimension of the space $V(K)$ is $12$, and the degrees of freedom
	\begin{equation}
		\left\{
		v(a_i),\ D_e^0(v),\ D_e^{n,0,K}(v),\ D_e^{n,1,K}(v):
		i=1,2,3,\ e\subset\partial K
		\right\}
		\label{eq:twelve-dofs-new}
	\end{equation}
	are unisolvent on $V(K)$, and $P_2(K)\subset V(K)$. For $v\in H^3(K)$, let $\Pi_Kv\in V(K)$ match these degrees of freedom. Then $\Pi_K$ is well defined, $\Pi_Kp=p$ for any $p\in P_2(K)$, and for $m=0,1,2$,
	\begin{equation}
		\|v-\Pi_Kv\|_{m,K}
		\lesssim h_K^{3-m}|v|_{3,K}.
		\label{eq:local-interpolation-new}
	\end{equation}
\end{proposition}
\begin{proof}
	The data \eqref{eq:twelve-dofs-new} reconstruct the omitted edge degrees of freedom as $D_e^1(v)=(v(z_0)-v(z_1))/12$ and $D_e^{n,2,K}(v)=0$. Since the inherited $18$ degrees of freedom are unisolvent on $S(K)$, they determine a unique member in $V(K)$, and thus $\dim V(K)=12$; moreover, $P_2(K)\subset V(K)$ because every quadratic function satisfies the reduction constraints.
	
	The embedding $H^3(K)\hookrightarrow C^1(\overline K)$ ensures that the interpolation data are well defined. After normalizing $K$ and scaling the conormal derivative functionals by $h_K$, the combined matrix of the reduction constraints and the retained degrees of freedom depends continuously on the normalized triangle geometry and is invertible by the inherited unisolvence on $S(K)$. Shape regularity therefore yields a uniform bound on its inverse. Quadratic reproduction, the Bramble--Hilbert lemma, and scaling give \eqref{eq:local-interpolation-new}.
\end{proof}

\begin{lemma}[Local harmonic kernel]
	\label{lem:local-harmonic-kernel-new}
	For any non-degenerate triangle $K$,
	\begin{equation}
		\{v\in V(K):\Delta_Kv=0\}
		=\{p\in P_2(K):\Delta_Kp=0\}.
		\label{eq:local-harmonic-kernel}
	\end{equation}
\end{lemma}
\begin{proof}
	It suffices to prove the inclusion from left to right. Write $v=q+\sum_{i=1}^3\vartheta_i b_i$, where $q\in P_4(K)$. Comparing the homogeneous cubic and quartic parts of $\Delta_Kv$ yields the $6\times3$ system analyzed in Appendix~\ref{app:local-kernel-bubbles}; its coefficient matrix has full column rank for every non-degenerate triangle. Therefore, all $\vartheta_i$ are zero, and $v\in P_4(K)$.
	
	The homogeneous cubic and quartic parts $v_3$ and $v_4$ are therefore harmonic. According to Appendix~\ref{app:local-kernel-high-order}, the two reduction constraints on each of two distinct edges emanating from a vertex force their complex coefficients to vanish. Consequently, $v_3=v_4=0$, implying that $v\in P_2(K)$.
\end{proof}

\subsection{Global REQ space}
\label{sec:global-b4pm-spaces}

For an interior edge $e=K^+\cap K^-$, choose a common unit tangent vector $t_e$, and label $K^+$ such that its local boundary orientation aligns with $t_e$. Denote the start and end vertices by $z_0$ and $z_1$, respectively. When comparing moments across $e$, use the common parameterization $z_0+t(z_1-z_0)$ in the weight $m_r$. Let $n_e^\pm=n_{K^\pm,e}$ be the unit outer conormal vectors in the planes of the two elements. We write $[v]_e=v^+-v^-$ and $\{v\}_e=(v^++v^-)/2$, and define
\begin{equation}
	J_{n,e}(v)
	=\partial_{n_e^+}v^++\partial_{n_e^-}v^-,
	\qquad
	A_{n,e}(v)
	=\frac12\left(\partial_{n_e^+}v^+-\partial_{n_e^-}v^-\right).
	\label{eq:outward-normal-jump-new}
\end{equation}
Both conormal derivatives are evaluated relative to the outer conormal vector of their respective element. If $K^+$ and $K^-$ are coplanar, then $n_e^-=-n_e^+$, and $J_{n,e}$ and $A_{n,e}$ reduce to the standard jump and average of normal derivatives, respectively. Figure~\ref{fig:polyhedral-edge-geometry} illustrates the geometry when the two elements are not coplanar.

\begin{figure}[htbp]
	\centering
	\tikzalt{Two noncoplanar planar triangles, K-minus and K-plus, meet along
		the edge e. Arrows indicate their element normals and the outward conormals
		used to define conormal jumps and averages; the two conormals need not be
		opposite.}
	\begin{tikzpicture}[scale=0.85,>=Latex,font=\small]
		\coordinate (A) at (0,1);
		\coordinate (B) at (0.5,-1);
		\coordinate (C) at (-3,-2);
		\coordinate (D) at (3,-1.5);
		\coordinate (E) at (4,0);
		\coordinate (F) at (1.5,2.5);
		\coordinate (M) at (0.125,0.5);
		\coordinate (G) at (0.25,0);
		
		\fill[blue!10] (A)--(F)--(E)--(B)--cycle;
		\draw[line width=1pt] (A)--(C)--(B)--cycle;
		\draw[line width=1pt] (A)--(B)--(D)--cycle;
		\draw[very thick] (A)--(B);
		\node at (-1.05,-0.72) {$K^-$};
		\node at (1.25,-0.58) {$K^+$};
		\node[left] at (G) {$e$};
		
		\draw[->,line width=1pt,orange!80!black] (M)--(-0.85,1)
		node[midway,above,text=black] {$n_e^+$};
		\draw[->,line width=1pt,blue!70!black] (M)--(1.25,0.95)
		node[midway,above,text=black] {$n_e^-$};
		\draw[->,line width=1pt,blue!70!black] (-1.55,-0.92)--(-2.12,-0.02)
		node[left,text=black] {$\nu_{K^-}$};
		\draw[->,line width=1pt,orange!80!black] (1.58,-0.72)--(2.22,0.16)
		node[right,text=black] {$\nu_{K^+}$};
	\end{tikzpicture}
	\tikzaltend
	\caption{Adjacent piecewise planar triangles sharing an edge $e$.}
	\label{fig:polyhedral-edge-geometry}
\end{figure}
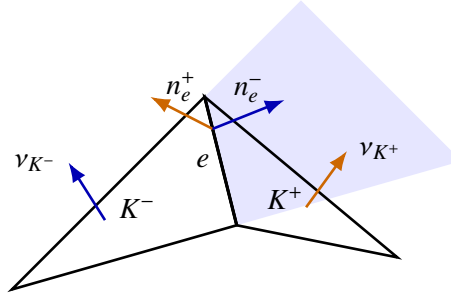

Based on this notation, define
\begin{equation}
	\begin{aligned}
		V_h(\mathcal M_h):=\bigl\{v_h:\;&
		v_h|_K\in V(K)\quad\forall K\in\cT_h(\mathcal M_h),\\
		&v_h\text{ is single-valued at each vertex},\\
		&\int_e[v_h]p\dd s_h=0
		\quad\forall p\in P_1(e),\ \forall e\in\cE_{h,{\rm int}}(\mathcal M_h),\\
		&\int_eJ_{n,e}(v_h)q\dd s_h=0
		\quad\forall q\in P_2(e),\ \forall e\in\cE_{h,{\rm int}}(\mathcal M_h)
		\bigr\}.
	\end{aligned}
	\label{eq:global-space-new}
\end{equation}
Under the local reduction constraints, a set of global degrees of freedom is given by the function value at each vertex, the zeroth-order function moment on each edge, and the zeroth- and first-order conormal derivative moments. Therefore, $\dim V_h(\mathcal M_h)=\#\cV_h(\mathcal M_h)+3\#\cE_h(\mathcal M_h)$.

\begin{remark}[Standard assembly]
	\label{rem:standard-assembly}
	All conditions in \eqref{eq:global-space-new} involve scalar data only: value moments use restrictions of $v_h$, and each conormal derivative is evaluated relative to the conormal vector of the element to which that restriction belongs. Since every element of $\mathcal M_h$ is a planar triangle and the same local construction is applied independently on each element, the contributions to the broken Laplacian stiffness matrix take the same cell-wise form as for a planar triangulation. The element stiffness matrices are computed from the local planar geometry determined by the 3D vertex coordinates. No identification of tangent planes or vertex derivatives is performed during assembly.
\end{remark}

In the planar case $\mathcal M_h=\Omega\subset\R^2$, $V_{h,0}\subset V_h(\Omega)$ is defined by setting $v_h(z)=0$ at each boundary vertex and $D_e^0(v_h)=D_e^{n,0,K}(v_h)=D_e^{n,1,K}(v_h)=0$ on each boundary edge $e\subset\partial K\cap\partial\Omega$, where $n_{K,e}$ is the exterior normal of $\Omega$. For a closed surface $\mathcal M_h$, define $\mathring V_h(\mathcal M_h)=\{v_h\in V_h(\mathcal M_h):\int_{\mathcal M_h}v_h\dd\sigma_h=0\}$ to eliminate the constant kernel.

\begin{proposition}[Global harmonic kernel]
	\label{prop:global-harmonic-kernel}
	On every connected closed piecewise planar triangulation $\mathcal M_h$ with non-degenerate triangles,
	\[
	\{v_h\in V_h(\mathcal M_h):\Delta_hv_h=0\}=\mathbb R.
	\]
	Consequently, $(\Delta_h\cdot,\Delta_h\cdot)_h$ is positive definite on $\mathring V_h(\mathcal M_h)$ without requiring any smallness condition on the mesh size.
\end{proposition}
\begin{proof}
	If $\Delta_hv_h=0$, Lemma~\ref{lem:local-harmonic-kernel-new} implies that $v_h|_K$ is a harmonic quadratic function on each triangle. On each edge, the quadratic function $[v_h]$ vanishes at both endpoints and has zero integral, so $[v_h]=0$. Furthermore, $J_{n,e}(v_h)\in P_1(e)$ and is orthogonal to $P_1(e)$, so $J_{n,e}(v_h)=0$. Hence, cell-wise integration by parts yields
	\[
	|v_h|_{1,h}^2=-(v_h,\Delta_hv_h)_h
	+\sum_{e\in\cE_h(\mathcal M_h)}(v_h,J_{n,e}(v_h))_e=0,
	\]
	where the edge values of $v_h$ are single-valued. Therefore, $v_h$ is constant on each triangle, and continuity and connectivity make it a global constant. The converse is obvious. Zero-mean constants are zero, thus proving positive definiteness. Uniform control over the broken $H^2$ norm will require additional assumptions from Section~\ref{sec:planar-surface-discrete-mt-new}.
\end{proof}

\subsection{Vertex residuals and rigid unfolding}
\label{sec:intrinsic-vertex-residual}

Let $z$ be an interior vertex of $\mathcal M_h$. According to the compatible orientations, enumerate the triangles in $\omega_z$ as $K_1,\ldots,K_m$, where $K_r$ and $K_{r+1}$ share the radial edge $e_r$, and $K_{m+1}=K_1$, so that the positive boundary orientation of $K_r$ points toward $z$ along $e_r$. Let $t_r$ be the corresponding unit tangent vector, and let $\eta_r^-=n_{K_r,e_r}$ and $\eta_r^+=-n_{K_{r+1},e_r}$. Let $\mathcal Q_r$ denote the rotation about $e_r$ that rigidly rotates the plane of $K_{r+1}$ into the plane of $K_r$. Then
\begin{equation}
	\mathcal Q_rt_r=t_r,
	\qquad
	\mathcal Q_r\eta_r^+=\eta_r^-.
	\label{eq:edge-rigid-rotation-new}
\end{equation}
For any piecewise function satisfying $v_h|_{K_r}\in V(K_r)$, let $g_r=\nabla_{K_r}v_h(z)$. Take the plane of $K_1$ as the unfolding plane $\mathcal P_z$, and define recursively
\begin{equation}
	U_1=I,
	\qquad
	U_{r+1}=U_r\mathcal Q_r\quad(1\leq r<m),
	\qquad
	\widetilde g_r=U_rg_r.
	\label{eq:accumulated-rotation-new}
\end{equation}
Thus, $U_r$ rigidly unfolds $K_r$ into $\mathcal P_z$. Let $\widetilde K_r=U_rK_r$ and transport $v_h|_{K_r}$ to $\widetilde K_r$ via $U_r$. Rigid invariance implies that its vertex gradient is $\widetilde g_r$. Transporting across the final radial edge completes a circuit and yields the holonomy rotation acting on the unfolding plane:
\begin{equation}
	\mathcal U_z:=U_m\mathcal Q_m:
	\mathcal P_z\longrightarrow\mathcal P_z.
	\label{eq:holonomy-rotation-new}
\end{equation}
Take $z$ as the common origin of all element planes. Since each $\mathcal Q_r$ preserves the compatible face orientations, $\mathcal U_z$ is a rotation in $\mathcal P_z$. Measuring angles in the orientation induced by $K_1$, we have
\begin{equation}
	\operatorname{ang}(\mathcal U_z)
	=\sum_{r=1}^m\theta_{r,z}-2\pi
	=-\kappa_z\pmod{2\pi}.
	\label{eq:holonomy-angle-defect-new}
\end{equation}
This sign convention is opposite to the forward transport convention in \cite[Proposition~2.5]{Christiansen2024Regge}. The two-stage construction of successive rigid unfolding in \eqref{eq:edge-rigid-rotation-new}--\eqref{eq:holonomy-angle-defect-new} is illustrated in Figure~\ref{fig:face-rotation-unfolding}.

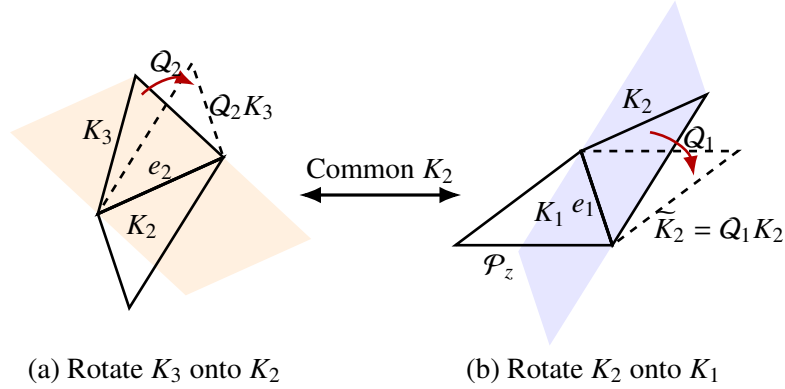
\begin{figure}[!htbp]
	\centering
	\tikzalt{Two-stage rigid unfolding of three adjacent triangles. First, Q2
		rotates K3 about e2 into the plane of K2; then Q1 rotates the resulting pair
		about e1 into the plane of K1, giving the composite map U3 equals Q1 Q2.}
	\begin{tikzpicture}[scale=0.83,>=Latex,font=\small]
		\begin{scope}[xshift=-0.5cm]
			\coordinate (a) at (0,-1);
			\coordinate (b) at (-0.5,0.5);
			\coordinate (c)	at (1.5,1.4);
			\coordinate (d)	at (1,2.9);
			\coordinate (e)	at (0.1,2.7);
			
			\fill[orange!12] (0.9,-0.8) -- (2.9,0.1) -- (0.1,2.7) -- (-1.9,1.8) -- cycle;
			
			\draw[line width=1pt] (a) -- (b) -- (c) -- cycle;
			\node[font=\small] at (0.2,0.3) {$K_2$};
			
			\draw[line width=1pt, dashed] (d) -- (b) -- (c) -- cycle;
			\node[font=\small] at (1.8,2.2) {$\mathcal Q_2K_3$};
			
			\draw[line width=1pt] (e) -- (b) -- (c) -- cycle;
			\node[font=\small] at (-0.5,1.8) {$K_3$};
			
			\draw[very thick] (b) -- (c);
			\node[font=\small] at (0.5,1.2) {$e_2$};
			
			\draw[->, line width=1pt, red!70!black]
			(0.2,2.4) to[bend left=35] (1.05,2.55);
			\node[font=\small] at (0.6,2.9) {$\mathcal Q_2$};
			
			\node at (0.4,-2) {(a) Rotate $K_3$ onto $K_2$};
		\end{scope}
		
		\draw[<->,very thick] (2.2,0.8) -- (4.8,0.8)
		node[midway,above] {Common $K_2$};
		
		\begin{scope}[xshift=7.2cm]
			\coordinate (z)      at (0,0);
			\coordinate (p1)     at (-2.5,0);
			\coordinate (p2)     at (-0.5,1.5);
			\coordinate (p3orig) at (1.5,2.4);
			\coordinate (p3hat)  at (2,1.5); 
			
			\fill[blue!10] (-1,-1.6) -- (1.5,2.4) -- (1,3.9) -- (-1.5,-0.1) -- cycle;
			
			\draw[line width=1pt] (z) -- (p2) -- (p3orig) -- cycle;
			\node[font=\small] at (0.4,2.3) {$K_2$};
			
			\draw[line width=1pt] (z) -- (p1) -- (p2) -- cycle;
			\node[font=\small] at (-1,0.5) {$K_1$};
			
			\draw[line width=1pt, dashed] (z) -- (p2) -- (p3hat) -- cycle;
			\node[font=\small] at (1.7,0.3) {$\widetilde K_2=\mathcal Q_1K_2$};
			
			\draw[very thick] (z) -- (p2);
			\node[font=\small] at (-0.45,0.6) {$e_1$};
			
			\draw[->, line width=1pt, red!70!black]
			(0.6,1.8) to[bend left=35] (1.3,1.1);
			\node[font=\small] at (1.4,1.7) {$\mathcal Q_1$};
			
			\node[font=\small] at (-1.8,-0.35) {$\mathcal P_z$};
			
			\node at (-0.3,-2) {(b) Rotate $K_2$ onto $K_1$};
		\end{scope}
	\end{tikzpicture}
	\tikzaltend
	\caption{Construction of $U_3=\mathcal Q_1\mathcal Q_2$ via successive rigid unfolding.}
	\label{fig:face-rotation-unfolding}
\end{figure}

In what follows, $\det(a,b)$ denotes the two-dimensional determinant in the oriented plane containing $a$ and $b$; its value is invariant under orientation-preserving rigid rotations. Define
\begin{equation}
	A_z^{\rm flat}(v_h)
	:=\sum_{r=1}^m\|\Delta_{K_r}v_h\|_{0,K_r}^2,
	\qquad
	Q_z^{\rm flat}(v_h)
	:=\sum_{r=1}^{m-1}
	\det(\widetilde g_{r+1},\widetilde g_r)
	+\det(\widetilde g_1,\widetilde g_m).
	\label{eq:planar-fan-forms}
\end{equation}
For a closed unfolded fan, let $\widetilde g_{m+1}:=\widetilde g_1$ so that $Q_z^{\rm flat}$ becomes a cyclic sum of its endpoint contributions. Surface fans need not be closed: the actual contribution of the final edge is $\det(\mathcal U_z\widetilde g_1,\widetilde g_m)$ rather than $\det(\widetilde g_1,\widetilde g_m)$, and their difference is the holonomy correction term in Proposition~\ref{lem:planar-vertex-residual}.

\begin{remark}[Boundary patch unfolding]
	\label{rem:boundary-patch-unfolding}
	For a boundary vertex $z$, order the associated triangles $K_1,\ldots,K_m$ from one boundary edge to the other according to compatible orientations. Take the plane of $K_1$ as $\mathcal P_z$, and define $g_r$, $\mathcal Q_r$, $U_r$, and $\widetilde g_r$ along the $m-1$ internal radial edges exactly as above. No closure rotation is introduced.
\end{remark}

\begin{proposition}[Edge-to-vertex residual decomposition]
	\label{lem:planar-vertex-residual}
	For any piecewise planar triangulation $\mathcal M_h$ and $v_h\in V_h(\mathcal M_h)$, the residual decomposition reads
	\begin{equation}
		\|\Delta_hv_h\|_{0,h}^2
		-\|D_h^2v_h\|_{0,h}^2
		=\sum_{z\in\cV_{h,{\rm int}}(\mathcal M_h)}Q_z^{\mathcal M_h}(v_h)
		+\sum_{z\in\cV_{h,\partial}(\mathcal M_h)}
		Q_{z,\partial}^{\mathcal M_h}(v_h),
		\label{eq:piecewise-flat-residual-decomposition}
	\end{equation}
	where for each interior vertex $z$,
	\begin{equation}
		Q_z^{\mathcal M_h}(v_h)
		:=Q_z^{\rm flat}(v_h)
		+\det\bigl((\mathcal U_z-I)\widetilde g_1,\widetilde g_m\bigr),
		\label{eq:holonomy-decomposition}
	\end{equation}
	and for each boundary vertex $z$,
	\[
	Q_{z,\partial}^{\mathcal M_h}(v_h)
	:=\sum_{r=1}^{m-1}\det(\widetilde g_{r+1},\widetilde g_r)
	+\frac12
	\sum_{\substack{e\in\cE_{h,\partial}(\mathcal M_h)\\ z\in e}}
	\int_e\bigl(\partial_{tt}v_h\,\partial_nv_h
	-\partial_{tn}v_h\,\partial_tv_h\bigr)\dd s_h.
	\]
\end{proposition}
\begin{proof}
	On each edge $e$ of $\partial K$, let $t$ be the induced unit tangent vector and $n=n_{K,e}$ the outer conormal vector. On each planar triangle, the identity
	\[
	(\Delta v)^2-|D^2v|^2
	=\operatorname{div}(\Delta v\nabla v-D^2v\nabla v)
	\]
	gives
	\begin{equation}
		\|\Delta v\|_{0,K}^2-\|D^2v\|_{0,K}^2
		=\int_{\partial K}
		\bigl(\partial_{tt}v\,\partial_nv
		-\partial_{tn}v\,\partial_tv\bigr)\dd s_h.
		\label{eq:planar-element-MT}
	\end{equation}
	For an interior edge $e=K^+\cap K^-$, rigidly align the two element planes around $e$ and use the common edge orientation. The jump product identity and \eqref{eq:outward-normal-jump-new}, followed by integration by parts along the edge and the vanishing of $[v_h]$ at the endpoints, yield the edge contribution
	\[
	\begin{aligned}
		2\int_eJ_{n,e}(v_h)\{\partial_{tt}v_h\}\dd s_h
		+2\int_e[v_h]\partial_{tt}A_{n,e}(v_h)\dd s_h
		+\left([\partial_tv_h]A_{n,e}(v_h)
		-J_{n,e}(v_h)\{\partial_tv_h\}\right)\Big|_{z_0}^{z_1}.
	\end{aligned}
	\]
	The enriched bubble functions vanish on $e$, so both value traces belong to $P_4(e)$ and $\{\partial_{tt}v_h\}\in P_2(e)$. Consequently, the first integral term vanishes because $J_{n,e}(v_h)\perp P_2(e)$, leaving
	\begin{equation}
		2\int_e[v_h]\partial_{tt}A_{n,e}(v_h)\dd s_h
		+\left(
		[\partial_tv_h]A_{n,e}(v_h)
		-J_{n,e}(v_h)\{\partial_tv_h\}
		\right)\Big|_{z_0}^{z_1}.
		\label{eq:planar-interior-edge-endpoints}
	\end{equation}
	Appendix~\ref{app:edge-cancellation} proves that the integrals in \eqref{eq:planar-interior-edge-endpoints} vanish for every $v_h\in V_h(\mathcal M_h)$. For the endpoint terms, rigidly transport the gradient on $K_{r+1}$ into the plane of $K_r$. Evaluation in the common edge coordinate system $(t_r,\eta_r^-)$, with each jump equal to the $K_r$ trace minus the transported $K_{r+1}$ trace, shows that the contribution of $e_r$ at $z$ is
	\begin{equation}
		\Bigl(
		[\partial_tv_h]\{\partial_\eta v_h\}
		-[\partial_\eta v_h]\{\partial_tv_h\}
		\Bigr)\Big|_z
		=\det(\mathcal Q_rg_{r+1},g_r).
		\label{eq:endpoint-determinant-new}
	\end{equation}
	Summing \eqref{eq:endpoint-determinant-new} over $r$ (setting $g_{m+1}:=g_1$), transporting the determinants to $\mathcal P_z$, and combining with \eqref{eq:planar-fan-forms} and definition \eqref{eq:holonomy-decomposition} yields
	\begin{equation}
		\sum_{r=1}^m\det(\mathcal Q_rg_{r+1},g_r)
		=
		\sum_{r=1}^{m-1}
		\det(\widetilde g_{r+1},\widetilde g_r)
		+\det(\mathcal U_z\widetilde g_1,\widetilde g_m)
		=Q_z^{\mathcal M_h}(v_h).
		\label{eq:surface-cycle-residual}
	\end{equation}
	
	For a boundary vertex $z$, the same unfolding along its $m-1$ internal radial edges gives the sum of determinants in the definition of $Q_{z,\partial}^{\mathcal M_h}(v_h)$, while the associated boundary edges give the integral terms. Summing over all vertex patches proves \eqref{eq:piecewise-flat-residual-decomposition}.
\end{proof}

\begin{proposition}[Vertex patch residual bound]
	\label{lem:surface-patch-residual-new}
	Let $v_h\in V_h(\mathcal M_h)$. At each interior vertex $z$,
	\begin{equation}
		|Q_z^{\mathcal M_h}(v_h)|
		\lesssim
		\sum_{K\subset\omega_z}\|\Delta_Kv_h\|_{0,K}^2
		+\frac{|\kappa_z|}{h_z^2}
		\sum_{K\subset\omega_z}\|\nabla_Kv_h\|_{0,K}^2.
		\label{eq:intrinsic-local-residual-bound-new}
	\end{equation}
	If $\mathcal M_h=\Omega\subset\R^2$, $v_h\in V_{h,0}$, and $z\in\cV_{h,\partial}(\Omega)$, then
	\begin{equation}
		|Q_{z,\partial}^{\Omega}(v_h)|
		\lesssim\sum_{K\subset\omega_z}\|\Delta_Kv_h\|_{0,K}^2.
		\label{eq:flat-boundary-patch-bound}
	\end{equation}
\end{proposition}
\begin{proof}
	We first estimate the flat part of \eqref{eq:holonomy-decomposition} on the enlarged space. Let $X_z$ consist of tuples $(v_r)_{r=1}^m\in\prod_{r=1}^mV(\widetilde K_r)$ satisfying the compatibility of vertex values and edge moments across the $m-1$ consecutive radial interfaces. No compatibility is imposed between the two boundary edges of this open fan. Even if their planar images overlap (which can occur when $\kappa_z<0$), the triangles remain as labeled sectors. The rigid unfolding restriction of $v_h$ to $\omega_z$ belongs to $X_z$. Both forms in \eqref{eq:planar-fan-forms} are defined on $X_z$, with the cyclic closure term in $Q_z^{\rm flat}$ understood algebraically.
	
	If $q\in X_z$ and $A_z^{\rm flat}(q)=0$, Lemma~\ref{lem:local-harmonic-kernel-new} implies that each component is a harmonic quadratic polynomial. The conditions on vertex values, function moments, and conormal moments on the $m-1$ internal radial edges imply that $[q]=[\partial_nq]=0$, so these components are restrictions of a single harmonic quadratic polynomial on the unfolded open fan:
	\begin{equation}
		\ker(A_z^{\rm flat}|_{X_z})
		=\left\{(p|_{\widetilde K_r})_{r=1}^m:
		p\in P_2(\mathcal P_z),\ \Delta_{\mathcal P_z}p=0\right\}.
		\label{eq:flat-open-fan-kernel}
	\end{equation}
	For the symmetric bilinear form $B_z^{\rm flat}$ associated with $Q_z^{\rm flat}$, every kernel polynomial $p$ has the same vertex gradient on each sector. The cyclic determinant sum in \eqref{eq:planar-fan-forms} therefore gives $B_z^{\rm flat}(v_h,p)=0$, so $B_z^{\rm flat}$ descends to the quotient by the kernel in \eqref{eq:flat-open-fan-kernel}. The norm equivalence established in Appendix~\ref{app:uniform-open-fan} is uniform over normalized admissible patches with uniformly bounded valence. Its proof uses compactness of the normalized patch geometries for each fixed valence and continuity of the harmonic-kernel projection. Hence,
	\begin{equation}
		|Q_z^{\rm flat}(v_h)|
		\lesssim A_z^{\rm flat}(v_h)
		=\sum_{K\subset\omega_z}\|\Delta_Kv_h\|_{0,K}^2.
		\label{eq:flat-open-fan-bound}
	\end{equation}
	
	The holonomy is a planar rotation through $-\kappa_z$, so
	\begin{equation}
		\|\mathcal U_z-I\|
		=2\left|\sin\frac{\kappa_z}{2}\right|\le|\kappa_z|.
		\label{eq:holonomy-rotation-norm-new}
	\end{equation}
	A vertex inverse estimate then bounds the correction term in \eqref{eq:holonomy-decomposition} as
	\[
	\left|\det((\mathcal U_z-I)\widetilde g_1,\widetilde g_m)\right|
	\lesssim \frac{|\kappa_z|}{h_z^2}
	\sum_{K\subset\omega_z}\|\nabla_Kv_h\|_{0,K}^2.
	\]
	Combining with \eqref{eq:flat-open-fan-bound}, this proves \eqref{eq:intrinsic-local-residual-bound-new}.
	
	In the planar homogeneous boundary case, the vanishing of the piecewise Laplacian energy similarly yields a single harmonic quadratic function $q$. The boundary degrees of freedom force $q|_e=0$ and $\partial_nq|_e=0$ on each boundary edge, and hence $q\equiv0$. Uniform norm equivalence on the normalized boundary patch spaces (as in Appendix~\ref{app:uniform-open-fan}) then proves \eqref{eq:flat-boundary-patch-bound}.
\end{proof}

\subsection{Discrete Miranda--Talenti-type estimates}
\label{sec:planar-surface-discrete-mt-new}

\begin{theorem}[Discrete Miranda--Talenti-type estimate on planar meshes]
	\label{thm:planar-discrete-MT}
	For every $v_h\in V_{h,0}$,
	\begin{equation}
		\|D_h^2v_h\|_{0,h}
		\lesssim\|\Delta_hv_h\|_{0,h},
		\qquad
		\|\Delta_hv_h\|_{0,h}
		\le\sqrt2\|D_h^2v_h\|_{0,h}.
		\label{eq:planar-discrete-MT}
	\end{equation}
\end{theorem}
\begin{proof}
	On planar meshes, $\mathcal U_z=I$ and $\kappa_z=0$. Propositions \ref{lem:planar-vertex-residual}--\ref{lem:surface-patch-residual-new} (including \eqref{eq:flat-boundary-patch-bound}) and finite patch overlap yield $\left|
	\|\Delta_hv_h\|_{0,h}^2-\|D_h^2v_h\|_{0,h}^2
	\right|
	\lesssim\|\Delta_hv_h\|_{0,h}^2$. Rearranging and using the pointwise inequality $|\Delta v|\le\sqrt2|D^2v|$ proves \eqref{eq:planar-discrete-MT}.
\end{proof}

For a uniformly shape-regular and quasi-uniform family $\{\mathcal M_h\}_{h\to0}$ of closed piecewise planar surface meshes with uniformly bounded vertex valences, assume that there exists a constant $C_{\rm PF}$ independent of $h$ such that every zero-mean $v\in H_h^1(\mathcal M_h)$ satisfies
\begin{equation}
	\|v\|_{0,h}
	\leq C_{\rm PF}\left(
	|v|_{1,h}+
	\left(
	\sum_{e\in\cE_h(\mathcal M_h)}h_e^{-1}\|[v]\|_{0,e}^2
	\right)^{1/2}
	\right).
	\label{eq:abstract-broken-poincare}
\end{equation}
Assume further that the discrete curvature scale of the mesh family is uniformly bounded:
\begin{equation}
	\Lambda_\kappa^*
	:=\sup_h\max_{z\in\cV_h(\mathcal M_h)}
	\frac{|\kappa_z|}{h_z^2}<\infty.
	\label{eq:abstract-angle-defect-bound}
\end{equation}

\begin{theorem}[Discrete Miranda--Talenti-type estimate on closed piecewise planar surfaces]
	\label{thm:discrete-miranda-talenti-new}
	Let the family $\{\mathcal M_h\}_{h\to0}$ satisfy \eqref{eq:abstract-angle-defect-bound} and \eqref{eq:abstract-broken-poincare}. Then there exist constants $h_0>0$ and $C_{\rm MT}>0$, independent of $h$ and depending possibly on $C_{\rm PF}$, $\Lambda_\kappa^*$, the mesh shape-regularity and quasi-uniformity constants, and the upper bound on the vertex valences, such that every $v_h\in V_h(\mathcal M_h)$ satisfies the following inequalities for $0<h\leq h_0$:
	\begin{equation}
		\|D_h^2v_h\|_{0,h}
		\leq C_{\rm MT}\|\Delta_hv_h\|_{0,h},
		\qquad
		\|\Delta_hv_h\|_{0,h}
		\leq\sqrt2\|D_h^2v_h\|_{0,h}.
		\label{eq:laplace-hessian-equivalence-new}
	\end{equation}
	Furthermore, every $v_h\in\mathring V_h(\mathcal M_h)$ satisfies
	\begin{equation}
		\|v_h\|_{2,h}
		\leq C_{\rm MT}\|\Delta_hv_h\|_{0,h}.
		\label{eq:discrete-MT-new}
	\end{equation}
\end{theorem}
\begin{proof}
	Subtracting a constant neither changes the gradient, the Hessian, the Laplacian, nor the edge jumps; hence we can take $v_h\in\mathring V_h(\mathcal M_h)$ when proving \eqref{eq:laplace-hessian-equivalence-new}. Since $\mathcal M_h$ is closed, the boundary sum in \eqref{eq:piecewise-flat-residual-decomposition} vanishes, and \eqref{eq:abstract-angle-defect-bound}, \eqref{eq:intrinsic-local-residual-bound-new}, and finite patch overlap yield
	\begin{equation}
		\|D_h^2v_h\|_{0,h}^2
		\lesssim
		\|\Delta_hv_h\|_{0,h}^2+
		\Lambda_\kappa^*|v_h|_{1,h}^2.
		\label{eq:hessian-prebound-new}
	\end{equation}
	Vertex single-valuedness, edge-moment continuity, the one-dimensional Poincaré inequality, and the Bramble--Hilbert lemma give
	\begin{align}
		\|[v_h]\|_{0,e}
		&\lesssim h_e^{3/2}\|D_h^2v_h\|_{0,\omega_e},
		\label{eq:value-jump-basic-new}\\
		\|J_{n,e}(v_h)\|_{0,e}
		&\lesssim h_e^{1/2}\|D_h^2v_h\|_{0,\omega_e}.
		\label{eq:normal-jump-basic-new}
	\end{align}
	
	Combining \eqref{eq:value-jump-basic-new} with \eqref{eq:abstract-broken-poincare} and finite patch overlap implies that
	\begin{equation}
		\|v_h\|_{0,h}
		\lesssim C_{\rm PF}
		\bigl(|v_h|_{1,h}+h\|D_h^2v_h\|_{0,h}\bigr).
		\label{eq:broken-poincare-new}
	\end{equation}
	
	Cell-wise integration by parts yields
	\begin{equation}
		|v_h|_{1,h}^2
		=-(v_h,\Delta_hv_h)_h+E_h(v_h),
		\label{eq:gradient-ibp-new}
	\end{equation}
	where
	\[
	E_h(v_h)
	=\sum_{e\in\cE_h(\mathcal M_h)}\int_e
	\left(\{v_h\}J_{n,e}(v_h)
	+[v_h]A_{n,e}(v_h)\right)\dd s_h.
	\]
	Orthogonality $J_{n,e}(v_h)\perp P_2(e)$ and $[v_h]\perp P_1(e)$, together with standard edge projection estimates, \eqref{eq:value-jump-basic-new}--\eqref{eq:normal-jump-basic-new}, and finite overlap, yield
	\begin{equation}
		|E_h(v_h)|\lesssim h^2\|D_h^2v_h\|_{0,h}^2.
		\label{eq:edge-gradient-error-new}
	\end{equation}
	
	Combining \eqref{eq:broken-poincare-new} and \eqref{eq:edge-gradient-error-new} with \eqref{eq:gradient-ibp-new}, and applying Young's inequality, we obtain
	\begin{align*}
		|v_h|_{1,h}^2
		&\lesssim C_{\rm PF}
		\bigl(|v_h|_{1,h}+h\|D_h^2v_h\|_{0,h}\bigr)
		\|\Delta_hv_h\|_{0,h}
		+h^2\|D_h^2v_h\|_{0,h}^2\\
		&\leq \frac12|v_h|_{1,h}^2
		+C(1+C_{\rm PF}^2)\|\Delta_hv_h\|_{0,h}^2
		+Ch^2\|D_h^2v_h\|_{0,h}^2.
	\end{align*}
	Using \eqref{eq:hessian-prebound-new} in the last term and choosing $h_0\leq1$ such that $Ch_0^2\Lambda_\kappa^*\leq\frac14$, we obtain
	\[
	|v_h|_{1,h}
	\lesssim(1+C_{\rm PF})\|\Delta_hv_h\|_{0,h}.
	\]
	Substituting back into \eqref{eq:hessian-prebound-new} gives
	\[
	\|D_h^2v_h\|_{0,h}
	\lesssim
	\bigl(1+\sqrt{\Lambda_\kappa^*}(1+C_{\rm PF})\bigr)
	\|\Delta_hv_h\|_{0,h}.
	\]
	This establishes the first inequality in \eqref{eq:laplace-hessian-equivalence-new}. The reverse inequality follows directly from the pointwise relation $|\Delta v|\leq\sqrt2|D^2v|$. Finally, \eqref{eq:broken-poincare-new} leads to \eqref{eq:discrete-MT-new}.
\end{proof}

\section{Discretization and Error Analysis}
\label{sec:error-analysis-new}

\subsection{Continuous problem and surface approximation}
\label{sec:surface-geometry-new}

Let $\gamma$ be the $C^4$ closed surface specified in Section 1, with unit normal vector $\nu$ and tangential projection $P=I-\nu\otimes\nu$.
For a scalar $w$ and a tangential vector field $\bm g$, choose local smooth extensions $\widetilde w$ and $\widetilde{\bm g}$ and define
\[
\nabla_\gamma w=P\nabla\widetilde w,
\qquad
\operatorname{div}_\gamma\bm g
=\operatorname{tr}\bigl(P\nabla\widetilde{\bm g}P\bigr),
\qquad
\Delta_\gamma w=\operatorname{div}_\gamma\nabla_\gamma w.
\]
These definitions are independent of the choice of extensions. The covariant Hessian is given by $\nabla_\gamma^2w=P\nabla\widetilde{\bm g}P$, where $\bm g=\nabla_\gamma w$ and $\widetilde{\bm g}$ is a smooth local extension.

For $m\geq0$, let $H^m(\gamma)$ denote the surface Sobolev space with norm $\|\cdot\|_{m,\gamma}$ and seminorm $|\cdot|_{m,\gamma}$.
We retain the inner product notation $(\cdot,\cdot)_\gamma$ from Section 1 and extend the zero-mean notation to $\mathring H^m(\gamma)=\{v\in H^m(\gamma):\int_\gamma v\dd\sigma=0\}$; in particular, $\mathring L^2(\gamma)=\mathring H^0(\gamma)$.
The load functional $l(v)=(f,v)_\gamma$ also applies to $v\in L^2(\gamma)$.
For $q,z\in H^1(\gamma)$, Green's formula on a closed surface is written as
\begin{equation}
	\langle\Delta_\gamma q,z\rangle_{H^{-1}(\gamma),H^1(\gamma)}
	=-(\nabla_\gamma q,\nabla_\gamma z)_\gamma.
	\label{eq:surface-green}
\end{equation}
If $q\in H^2(\gamma)$, the left-hand term equals $(\Delta_\gamma q,z)_\gamma$.
Zero-mean elliptic estimates and the Lax--Milgram lemma lead to the unique solution of \eqref{eq:continuous-problem-new}.
The $C^4$ regularity of $\gamma$ yields
\begin{equation}
	u\in\mathring H^4(\gamma),
	\qquad
	\|u\|_{4,\gamma}\leq C_\gamma\|f\|_{0,\gamma},
	\label{eq:continuous-regularity-new}
\end{equation}
cf. \cite[eq.~(3.16)]{WuZhou2025}. Here and below, $C_\gamma$ denotes a general constant depending on the geometry and regularity of the fixed surface. For discrete estimates, it may also depend on uniform mesh regularity and the upper bound on vertex valences, but it is independent of $h$ and the function being estimated.

The continuous counterpart of the discrete stability estimate follows from \eqref{eq:surface-bochner-grisvard}. Let $K_\gamma^-:=\max\{-K_\gamma,0\}$ and let $C_{P,\gamma}$ be the Poincaré constant on $\gamma$. Subtracting the mean does not change derivatives, so Green's formula and the surface Poincaré inequality give $|w|_{1,\gamma}\leq C_{P,\gamma}\|\Delta_\gamma w\|_{0,\gamma}$.
Consequently, \eqref{eq:surface-bochner-grisvard} implies
\begin{equation}
	\|\nabla_\gamma^2w\|_{0,\gamma}
	\leq\left(1+C_{P,\gamma}^2\|K_\gamma^-\|_{L^\infty(\gamma)}\right)^{1/2}
	\|\Delta_\gamma w\|_{0,\gamma},
	\qquad w\in H^2(\gamma).
	\label{eq:continuous-weak-MT-new}
\end{equation}
Thus, the continuous Laplacian controls the covariant Hessian with a geometry-dependent constant. We next specify the interpolating surface meshes and verify the assumptions of Theorem~\ref{thm:discrete-miranda-talenti-new}.

Choose a sufficiently small $\delta>0$ such that the tubular neighborhood $U_\delta(\gamma)=\{x\in\R^3:\operatorname{dist}(x,\gamma)<\delta\}$ admits the signed distance function $d$, whose sign selection is compatible with $\nu$.
Extend $\nu$ and $P$ to $U_\delta(\gamma)$ via $\nu=\nabla d$ and $P=I-\nu\otimes\nu$. Let $H=\nabla^2d$ and define the closest-point projection $p(x)=x-d(x)\nu(x)$. Then $H|_\gamma$ is the Weingarten map and $\nabla p=P-dH$. Define the global curvature quantity $\mathcal K_0=\max_{i,j}\|H_{ij}\|_{L^\infty(\gamma)}$.
For a scalar, vector, or tensor field $w$ defined on $\gamma$, its closest-point extension is defined component-wise as $w^e=w\circ p$.

Let $\{\Gamma_h\}_{h\to0}$ be a shape-regular and quasi-uniform family of closed interpolating piecewise planar surfaces whose vertices lie on the fixed surface $\gamma$.
In the notation of Section 2, take $\mathcal M_h=\Gamma_h$ and abbreviate $\cT_h=\cT_h(\Gamma_h)$, $\cE_h=\cE_h(\Gamma_h)$, $\cV_h=\cV_h(\Gamma_h)$, $V_h=V_h(\Gamma_h)$, and $\mathring V_h=\mathring V_h(\Gamma_h)$. For the geometric estimates, uniform stability, and error bounds below, assume that $h$ is sufficiently small such that
\begin{equation}
	\Gamma_h\subset U_\delta(\gamma),
	\qquad
	p|_{\Gamma_h}:\Gamma_h\longrightarrow\gamma\ \text{is a bijection},
	\qquad
	h^2\mathcal K_0^2\leq c_*<1,
	\label{eq:surface-mesh-assumptions}
\end{equation}
where $c_*$ is sufficiently small and independent of $h$ and curvature.
For $v_h$ on $\Gamma_h$, define its lift as $v_h^\ell=v_h\circ(p|_{\Gamma_h})^{-1}$.

Let $K^\gamma=p(K)$ and define $\mathcal K_{0,K}=\max_{i,j}\|H_{ij}\|_{L^\infty(K^\gamma)}$ and $\mathcal K_{0,\omega}=\max_{K\subset\omega}\mathcal K_{0,K}$.
Orient the piecewise constant unit normal vectors of the discrete surface compatibly with $\nu^e$, and denote them by $\nu_h$. Let $\nu_K=\nu_h|_K$ and $P_h=I-\nu_h\otimes\nu_h$.

The cell-wise estimates for normal and area elements are
\begin{align}
	\|d\|_{L^\infty(K)}
	&\lesssim h_K^2\mathcal K_{0,K},
	&
	\|\nu^e-\nu_K\|_{L^\infty(K)}
	&\lesssim h_K\mathcal K_{0,K},
	\label{eq:geometry-normal}\\
	\dd\sigma(p(x))&=\mu_h(x)\dd\sigma_h(x),
	&
	\|1-\mu_h\|_{L^\infty(K)}
	+\|1-\mu_h^{-1}\|_{L^\infty(K)}
	&\lesssim h_K^2\mathcal K_{0,K}^2.
	\label{eq:geometry-area}
\end{align}
These estimates and the lifting norm equivalences below are standard; see \cite[section~2 and Lemma~3]{GuzmanMadureiraSarkisWalker2018}. For each $K\in\cT_h$, let $v^\ell=v\circ(p|_K)^{-1}$ denote the lift of a function $v$ defined on $K$ to $K^\gamma$. Then
\begin{align}
	\left|
	\|v\|_{0,K}^2-\|v^\ell\|_{0,K^\gamma}^2
	\right|
	&\lesssim h_K^2\mathcal K_{0,K}^2
	\|v^\ell\|_{0,K^\gamma}^2,
	\qquad v\in L^2(K),
	\label{eq:l2-norm-equivalence-curvature}\\
	\left|
	|v|_{1,K}^2-|v^\ell|_{1,K^\gamma}^2
	\right|
	&\lesssim h_K^2\mathcal K_{0,K}^2
	|v^\ell|_{1,K^\gamma}^2,
	\qquad v\in H^1(K).
	\label{eq:h1-seminorm-equivalence}
\end{align}
By \eqref{eq:surface-mesh-assumptions}, both estimates hold in either direction with constants independent of $h$.

The tangential gradient satisfies $|v_h|_{1,K}^2=(R_h\nabla_\gamma v_h^\ell, \nabla_\gamma v_h^\ell)_{K^\gamma}$, where $R_h=\mu_h^{-1}P(I-dH)P_h(I-dH)P$, and
\begin{equation}
	\|(R_h-I)P\|_{L^\infty(K^\gamma)}
	\lesssim h_K^2\mathcal K_{0,K}^2.
	\label{eq:gradient-transport}
\end{equation}
This identity and \eqref{eq:gradient-transport} (where discrete quantities are understood after lifting to $\gamma$) originate from \cite[proof of Lemma~10]{GuzmanMadureiraSarkisWalker2018}.
For $m=2,3$, we also use
\begin{equation}
	\|v_h^\ell\|_{m,K^\gamma}\leq C_\gamma\|v_h\|_{m,K},
	\qquad
	\|v_h\|_{m,K}\leq C_\gamma\|v_h^\ell\|_{m,K^\gamma},
	\qquad v_h\in H^m(K),
	\label{eq:lift-norm-equivalence-new}
\end{equation}
cf. \cite[eq.~(2.4)]{WuZhou2025}.
For $w\in H^3(\gamma)$, the geometric error of the Laplacian satisfies
\begin{equation}
	\|(\Delta_\gamma w)^e-\Delta_hw^e\|_{0,h}
	\leq C_\gamma\bigl(h|w|_{1,\gamma}
	+h^2|w|_{2,\gamma}\bigr).
	\label{eq:laplace-geometric-error}
\end{equation}
This estimate follows by density from \cite[appendix~C, eq.~(C.7)]{LarssonLarson2017}.

\begin{lemma}[Geometric admissibility of interpolating surfaces]
	\label{lem:uniform-poincare-discrete-surface}
	\label{lem:geometric-admissibility-new}
	For sufficiently small $h$, every zero-mean $v_h\in H^1(\Gamma_h)$ satisfies
	\begin{equation}
		\|v_h\|_{0,h}
		\leq C_\gamma|v_h|_{1,h}.
		\label{eq:uniform-poincare-discrete-surface}
	\end{equation}
	More generally, every zero-mean $v_h\in H_h^1(\Gamma_h)$ satisfies
	\begin{equation}
		\|v_h\|_{0,h}
		\leq C_\gamma\left(
		|v_h|_{1,h}
		+\left(\sum_{e\in\cE_h}h_e^{-1}\|[v_h]\|_{0,e}^2\right)^{1/2}
		\right).
		\label{eq:uniform-broken-poincare-surface}
	\end{equation}
	Furthermore, every $z\in\cV_h$ satisfies
	\begin{equation}
		|\kappa_z|\lesssim
		h_z^2\mathcal K_{0,\omega_z}^2.
		\label{eq:angle-defect-new}
	\end{equation}
	Therefore, $\max_z|\kappa_z|\to0$ and $\Lambda_\kappa^*\lesssim\mathcal K_0^2<\infty$, so Theorem~\ref{thm:discrete-miranda-talenti-new} applies to $\{\Gamma_h\}$.
\end{lemma}
\begin{proof}
	For zero-mean $v_h\in H^1(\Gamma_h)$, \eqref{eq:geometry-area} gives
	\[
	|\gamma|^{-1/2}\left|\int_\gamma v_h^\ell\dd\sigma\right|
	=|\gamma|^{-1/2}\left|\int_{\Gamma_h}v_h(\mu_h-1)\dd\sigma_h\right|
	\lesssim h^2\mathcal K_0^2\|v_h\|_{0,h}.
	\]
	Applying the surface Poincaré inequality to $v_h^\ell$ after subtracting its mean, using \eqref{eq:l2-norm-equivalence-curvature}--\eqref{eq:h1-seminorm-equivalence}, and absorbing the resulting $h^2\mathcal K_0^2\|v_h\|_{0,h}$ term proves \eqref{eq:uniform-poincare-discrete-surface}.
	
	For the broken estimate, we present on $\Gamma_h$ the averaging argument underlying \cite[Theorem~5.1]{Brenner2003}. Set
	\[
	J_h(v_h)^2:=\sum_{e\in\cE_h}h_e^{-1}\|[v_h]\|_{0,e}^2,
	\qquad
	\bar v_K:=|K|^{-1}\int_Kv_h\dd\sigma_h.
	\]
	Let $\mathcal B_hv_h$ be the continuous piecewise affine function whose value at each vertex is the arithmetic mean of $\bar v_K$ over its associated triangles.
	For adjacent triangles $K^+$ and $K^-$, the element Poincaré inequality and trace inequality yield
	\[
	|\bar v_{K^+}-\bar v_{K^-}|
	\lesssim |v_h|_{1,K^+}+|v_h|_{1,K^-}
	+h_e^{-1/2}\|[v_h]\|_{0,e}.
	\]
	Chaining this estimate across each vertex patch, using the uniform valence upper bound, and scaling on each planar triangle yields
	\[
	\|v_h-\mathcal B_hv_h\|_{0,h}+h|\mathcal B_hv_h|_{1,h}
	\lesssim h\bigl(|v_h|_{1,h}+J_h(v_h)\bigr).
	\]
	These local estimates are invariant under rigid rotations of triangles.
	Subtracting the mean of $\mathcal B_hv_h$ yields $\mathcal B^0_hv_h$. Since $v_h$ has zero mean, $\|v_h-\mathcal B^0_hv_h\|_{0,h}\leq2\|v_h-\mathcal B_hv_h\|_{0,h}$.
	Applying \eqref{eq:uniform-poincare-discrete-surface} to $\mathcal B^0_hv_h$ and taking $h\leq1$ proves \eqref{eq:uniform-broken-poincare-surface} with a constant independent of $h$.
	
	For the angle defect, let $\tau_r$ be the unit radial edge direction at $z$, let $P_z=I-\nu(z)\otimes\nu(z)$, and set $\alpha_r=\tau_r\cdot\nu(z)$ and $\widehat\tau_r=P_z\tau_r/|P_z\tau_r|$.
	Equation \eqref{eq:geometry-normal} gives $|\alpha_r|\lesssim h_z\mathcal K_{0,\omega_z}$. Thus
	\[
	\widehat\tau_r\cdot\widehat\tau_{r+1}
	=\frac{\tau_r\cdot\tau_{r+1}-\alpha_r\alpha_{r+1}}
	{\sqrt{1-\alpha_r^2}\sqrt{1-\alpha_{r+1}^2}}
	=\tau_r\cdot\tau_{r+1}+O(h_z^2\mathcal K_{0,\omega_z}^2).
	\]
	Shape regularity restricts triangle angles away from $0$ and $\pi$, so under this projection, the change in each angle is $O(h_z^2\mathcal K_{0,\omega_z}^2)$.
	Local chart representations of $\gamma$, the bijectivity of $p$, and the normal estimates ensure that for sufficiently small $h$, the projected radial edges form one cyclic fan around $z$ in its tangent plane. Its angles sum to $2\pi$. Summing the angle differences over the uniformly bounded number of incident triangles proves \eqref{eq:angle-defect-new}; see also \cite{BorrelliCazalsMorvan2003}. Taking the maximum over $z$ yields the stated convergence and the bound on $\Lambda_\kappa^*$. Only this bound on the angle defect is used; pointwise convergence of the normalized angle defect to the Gaussian curvature is not required.
\end{proof}

\subsection{Finite element discretization}
\label{sec:discrete-scheme-new}

Given $f\in\mathring L^2(\gamma)$, let $u\in\mathring H^2(\gamma)$ be the solution of \eqref{eq:continuous-problem-new}. Define the discrete zero-mean right-hand side $f_h=f^e-|\Gamma_h|^{-1}\int_{\Gamma_h}f^e\dd\sigma_h$.
The finite element problem reads: find $u_h\in\mathring V_h$ such that
\begin{equation}
	a_h(u_h,v_h)=l_h(v_h)
	\qquad\forall v_h\in\mathring V_h,
	\label{eq:discrete-problem-new}
\end{equation}
where $a_h$ is the bilinear form from \eqref{eq:intro-discrete-laplacian-form} restricted to $V_h\times V_h$, and $l_h:V_h\to\R$ is defined as $l_h(v_h)=(f_h,v_h)_h$.
The stiffness matrix of \eqref{eq:discrete-problem-new} is assembled according to Remark~\ref{rem:standard-assembly}; apart from the local planar element geometry determined by the 3D vertex coordinates, the only additional use of the exact surface in assembling the discrete problem is through the closest-point extension $f^e$ in the load vector.
The relationship between $f_h$ and $f^e$ implies that for every $v_h\in\mathring V_h$,
\begin{equation}
	l_h(v_h)=(f^e,v_h)_h,
	\qquad
	\|f_h\|_{0,h}
	\leq\|f^e\|_{0,h}
	\leq\bigl(1+Ch^2\mathcal K_0^2\bigr)\|f\|_{0,\gamma}.
	\label{eq:source-basic-new}
\end{equation}
The last inequality follows by summing \eqref{eq:l2-norm-equivalence-curvature} over the elements.

\begin{theorem}[Well-posedness of the discrete problem]
	\label{thm:discrete-wellposedness-new}
	On every connected closed triangulation $\Gamma_h$ with non-degenerate triangles where $f_h$ is defined, \eqref{eq:discrete-problem-new} has a unique solution $u_h\in\mathring V_h$. Under the geometric assumptions of Section~\ref{sec:surface-geometry-new}, for sufficiently small $h$, the following stability estimate is uniform with respect to $h$. Defining $\|f_h\|_{-2,h}
	:=\sup_{0\neq v_h\in\mathring V_h}
	\frac{|(f_h,v_h)_h|}{\|v_h\|_{2,h}}$, we have
	\begin{equation}
		\|u_h\|_{2,h}
		\leq C_\gamma\|f_h\|_{-2,h}.
		\label{eq:discrete-stability-new}
	\end{equation}
\end{theorem}
\begin{proof}
	Proposition~\ref{prop:global-harmonic-kernel} makes $a_h$ positive definite on the finite-dimensional space $\mathring V_h$, thereby proving unique solvability on each fixed mesh. For sufficiently refined interpolating meshes, \eqref{eq:discrete-MT-new} provides a coercivity constant independent of $h$.
	Testing with $u_h$ yields
	\[
	\|\Delta_hu_h\|_{0,h}^2
	\leq\|f_h\|_{-2,h}\|u_h\|_{2,h}
	\leq C_\gamma\|f_h\|_{-2,h}\|\Delta_hu_h\|_{0,h},
	\]
	from which cancellation and \eqref{eq:discrete-MT-new} prove \eqref{eq:discrete-stability-new}, with the zero solution case being trivial.
\end{proof}

\subsection{Global surface interpolation and approximation}
\label{sec:approximation-new}

For $w\in H^3(\gamma)$, define the broken local interpolation operator $\Pi_h$ element-wise by $(\Pi_hw^e)|_K=\Pi_K(w^e|_K)$.
Proposition~\ref{lem:local-interpolation-new}, \eqref{eq:l2-norm-equivalence-curvature}, \eqref{eq:h1-seminorm-equivalence}, and \eqref{eq:lift-norm-equivalence-new} lead to
\begin{equation}
	\|w^e-\Pi_hw^e\|_{m,h}
	\leq C_\gamma h^{3-m}\|w\|_{3,\gamma},
	\qquad m=0,1,2.
	\label{eq:broken-local-interpolation-estimate}
\end{equation}
For each edge $e=K^+\cap K^-$, use the fixed edge orientation $t_e$ and common parameter $t\in[0,1]$ from Section~\ref{sec:b4pm-space-mt}. For $a\in\cV_h$ and $r=0,1$, define the global interpolation operator $I_h:H^3(\gamma)\longrightarrow V_h$ via the following local degrees of freedom:
\begin{align}
	I_hw(a)&=w^e(a), \quad\quad
	D_e^0(I_hw)=D_e^0(w^e),
	\label{eq:global-interpolant-vertex}\\
	D_e^{n,r,K^\pm}(I_hw)
	&=\pm\frac1{2|e|}\int_e
	\left(\partial_{n_e^+}w^e-\partial_{n_e^-}w^e\right)m_r(t)\dd s_h.
	\label{eq:global-interpolant-normal}
\end{align}
In the second line, the upper and lower signs correspond to $K^+$ and $K^-$, respectively, and both conormal moments use the common parameter $t$.
Proposition~\ref{lem:local-unisolvence-new} ensures that these degrees of freedom uniquely determine a function in $V(K)$ on each element. Equations \eqref{eq:global-interpolant-vertex}--\eqref{eq:global-interpolant-normal} and the local reduction constraints imply global compatibility of vertex values, zeroth- and first-order function moments, and zeroth-, first-, and second-order conormal derivative moments.

For $w\in H^3(\gamma)$, let
\[
I_h^0w:=I_hw-\bar c_h(w),
\qquad
\bar c_h(w):=\frac1{|\Gamma_h|}
\int_{\Gamma_h}I_hw\dd\sigma_h,
\]
so that $I_h^0w\in\mathring V_h$.

\begin{proposition}[Zero-mean global interpolation]
	\label{thm:global-interpolation-new}
	For each $w\in\mathring H^3(\gamma)$,
	\begin{align}
		\|w^e-I_h^0w\|_{0,h}+|w^e-I_h^0w|_{1,h}
		&\leq C_\gamma h^2\|w\|_{3,\gamma},
		\label{eq:mean-zero-interpolation-H1-new}\\
		\|D_h^2(w^e-I_h^0w)\|_{0,h}
		+\|(\Delta_\gamma w)^e-\Delta_hI_h^0w\|_{0,h}
		&\leq C_\gamma h\|w\|_{3,\gamma}.
		\label{eq:mean-zero-interpolation-H2-new}
	\end{align}
\end{proposition}
\begin{proof}
	By construction, $I_hw$ and $\Pi_hw^e$ share vertex values and zeroth-order function moments; the reduction constraints identify their first-order function moments, and both have zero second-order conormal-derivative moments. Therefore, only the zeroth- and first-order conormal moments of $I_hw-\Pi_hw^e$ remain, and they satisfy
	\[
	|D_e^{n,r,K^\pm}(I_hw-\Pi_hw^e)|
	=\frac1{2|e|}\left|\int_e
	\left(\partial_{n_e^+}w^e+\partial_{n_e^-}w^e\right)
	m_r(t)\dd s_h\right|.
	\]
	Opposite local edge orientations give
	\begin{equation}
		n_e^++n_e^-=t_e\times(\nu_h^+-\nu_h^-).
		\label{eq:conormal-sum-cross}
		\end{equation}
	At each point of $e$, decompose $t_e$ and $\nu_h^+-\nu_h^-$ into tangential and normal components with respect to the closest point $p(x)$ on $\gamma$:
	\[
	t_e=t_T+\alpha\nu^e,
	\qquad
	\nu_h^+-\nu_h^-=q_T+\beta\nu^e,
	\qquad
	t_T=Pt_e,
	\quad q_T=P(\nu_h^+-\nu_h^-).
	\]
	Using $t_e\cdot\nu_h^+=0$, \eqref{eq:geometry-normal}, quasi-uniformity, and the unit vector identity $\nu^e\cdot\nu_h^\pm=1-\frac12|\nu^e-\nu_h^\pm|^2$, we have
	\[
	|\alpha|+|q_T|\lesssim h_e\mathcal K_{0,\omega_e},
	\qquad
	|\beta|\lesssim h_e^2\mathcal K_{0,\omega_e}^2.
	\]
	Since $P(t_T\times q_T)=0$ and each remaining term in \eqref{eq:conormal-sum-cross} contains $\beta$ or $\alpha q_T$, it follows that
	\begin{equation}
		\|P(n_e^++n_e^-)\|_{L^\infty(e)}
		\lesssim h_e^2\mathcal K_{0,\omega_e}^2.
		\label{eq:tangential-conormal-sum}
		\end{equation}
	Since $\nabla w^e=P\nabla w^e$, \eqref{eq:tangential-conormal-sum}, edge trace inequalities, and lifting norm equivalences yield
	\[
	|D_e^{n,r,K^\pm}(I_hw-\Pi_hw^e)|
	\lesssim
	\mathcal K_{0,\omega_e}^2\left(
	h_e\|\nabla_hw^e\|_{0,\omega_e}
	+h_e^2\|D_h^2w^e\|_{0,\omega_e}\right).
	\]
	
	The scaling of derivative data basis functions in $H^m(K)$ is $h_K^{2-m}$.
	Thus, finite edge patch overlap and \eqref{eq:broken-local-interpolation-estimate} imply
	\[
	\|w^e-I_hw\|_{m,h}
	\leq C_\gamma h^{3-m}\|w\|_{3,\gamma},
	\qquad m=0,1,2.
	\]
	The geometric Laplacian estimate \eqref{eq:laplace-geometric-error} then gives
	\[
	\|(\Delta_\gamma w)^e-\Delta_hI_hw\|_{0,h}
	\leq\|(\Delta_\gamma w)^e-\Delta_hw^e\|_{0,h}
	+\|\Delta_h(w^e-I_hw)\|_{0,h}
	\leq C_\gamma h\|w\|_{3,\gamma}.
	\]
	
	Since $I_h^0w-I_hw$ is a constant, the two interpolants have the same derivatives.
	Furthermore, since $w$ has zero mean, area transformation yields
	\[
	|\bar c_h(w)|
	\lesssim h^2\mathcal K_0^2\|w\|_{0,\gamma}
	+\|I_hw-w^e\|_{0,h}
	\leq C_\gamma h^2\|w\|_{3,\gamma}.
	\]
	This bound combined with previous estimates proves \eqref{eq:mean-zero-interpolation-H1-new} and \eqref{eq:mean-zero-interpolation-H2-new}.
\end{proof}

\subsection{Conforming companion operator and consistency estimates}
\label{sec:auxiliary-estimates-new}

Since $V(K)\subset P_6(K)$, define the nodal averaging operator $\mathcal{A}_h$ at the degree-six Lagrange nodes of each triangle as follows: preserve element values at interior nodes, average the two traces at interior edge nodes, and use single values at vertex nodes. Define the mean-preserving correction by:
\[
\widetilde{\mathcal{A}}_h v_h
=\mathcal{A}_h v_h+\frac1{|\Gamma_h|}
\int_{\Gamma_h}(v_h-\mathcal{A}_h v_h)\dd\sigma_h.
\]
By construction, $\mathcal{A}_hv_h$ is a continuous piecewise polynomial of degree six. The mean correction is constant, so $\widetilde{\mathcal{A}}_hv_h$ also belongs to $H^1(\Gamma_h)$ and has the same integral as $v_h$.

\begin{lemma}[Bounds for the mean-preserving conforming companion operator]
	\label{lem:h1-companion-new}
	For each $v_h\in V_h$,
	\begin{equation}
		\|v_h-\widetilde{\mathcal{A}}_h v_h\|_{0,h}
		+h|v_h-\widetilde{\mathcal{A}}_h v_h|_{1,h}
		\leq C_\gamma h^2\|\Delta_hv_h\|_{0,h}.
		\label{eq:companion-approx-new}
	\end{equation}
	If in addition $v_h\in\mathring V_h$, then
	\begin{equation}
		\|\widetilde{\mathcal{A}}_h v_h\|_{1,h}
		\leq C_\gamma\|\Delta_hv_h\|_{0,h}.
		\label{eq:companion-stability-new}
	\end{equation}
\end{lemma}
\begin{proof}
	On each $K$, $v_h-\mathcal{A}_hv_h$ vanishes at the interior Lagrange nodes and at the vertices, and equals (up to sign) half of the corresponding jump at interior edge nodes. Thus, finite-dimensional norm equivalence, edge inverse estimates, and scaling properties give
	\[
	\|v_h-\mathcal{A}_h v_h\|_{m,h}^2
	\lesssim
	\sum_{e\in\cE_h}h_e^{1-2m}\|[v_h]\|_{0,e}^2,
	\qquad m=0,1.
	\]
	See corresponding nodal averaging estimates on planar meshes \cite[Theorem~2.2]{KarakashianPascal2003}. Using \eqref{eq:value-jump-basic-new} and finite patch overlap gives
	\[
	\|v_h-\mathcal{A}_h v_h\|_{0,h}
	\lesssim h^2\|D_h^2v_h\|_{0,h},
	\qquad
	|v_h-\mathcal{A}_h v_h|_{1,h}
	\lesssim h\|D_h^2v_h\|_{0,h}.
	\]
	The mean correction satisfies $\|\widetilde{\mathcal{A}}_h v_h-\mathcal{A}_h v_h\|_{0,h}
	\leq\|v_h-\mathcal{A}_h v_h\|_{0,h}$ and leaves the $H^1$ seminorm unchanged.
	Previous estimates along with \eqref{eq:laplace-hessian-equivalence-new} therefore prove \eqref{eq:companion-approx-new}. If $v_h\in\mathring V_h$, then $\widetilde{\mathcal{A}}_h v_h$ also has zero mean, so \eqref{eq:uniform-poincare-discrete-surface}, \eqref{eq:companion-approx-new}, and $|v_h|_{1,h}\leq C_\gamma\|\Delta_hv_h\|_{0,h}$ yield \eqref{eq:companion-stability-new}.
\end{proof}

\begin{lemma}[Edge consistency estimate]
	\label{lem:edge-consistency-new}
	For each $w\in H^3(\gamma)$ and $v_h\in V_h$,
	\begin{equation}
		\left|
		\sum_{e\in\cE_h}
		((\Delta_\gamma w)^e,J_{n,e}(v_h))_e
		\right|
		\leq C_\gamma
		h\|w\|_{3,\gamma}\|\Delta_hv_h\|_{0,h}.
		\label{eq:edge-consistency-energy-new}
	\end{equation}
	For each $w,v\in H^3(\gamma)$,
	\begin{equation}
		\left|
		\sum_{e\in\cE_h}
		((\Delta_\gamma w)^e,J_{n,e}(I_h^0v))_e
		\right|
		\leq C_\gamma
		h^2\|w\|_{3,\gamma}\|v\|_{3,\gamma}.
		\label{eq:edge-consistency-interpolants-new}
	\end{equation}
\end{lemma}
\begin{proof}
	The moment conditions yield $J_{n,e}(v_h)\perp P_2(e)$ and $J_{n,e}(I_h^0v)\perp P_1(e)$. Scaled Poincaré inequalities and inverse trace estimates produce
	\[
	h_e^{-1/2}\|J_{n,e}(v_h)\|_{0,e}
	\lesssim \|D_h^2v_h\|_{0,\omega_e},
	\qquad
	h_e^{-1/2}\|J_{n,e}(I_h^0v)\|_{0,e}
	\lesssim h_e|I_h^0v|_{3,\omega_e}.
	\]
	Consequently, finite edge patch overlap and \eqref{eq:laplace-hessian-equivalence-new} give
	\[
	\left(\sum_{e\in\cE_h}h_e^{-1}
	\|J_{n,e}(v_h)\|_{0,e}^2\right)^{1/2}
	\leq C_\gamma\|\Delta_hv_h\|_{0,h}.
	\]
	For each $K$, choose $p_K\in P_2(K)$ such that $|v^e-p_K|_{2,K}\lesssim h_K|v^e|_{3,K}$. Inverse estimates yield
	\[
	|I_h^0v|_{3,K}=|I_h^0v-p_K|_{3,K}
	\lesssim h_K^{-1}|I_h^0v-p_K|_{2,K}
	\lesssim h_K^{-1}|I_h^0v-v^e|_{2,K}+|v^e|_{3,K}.
	\]
	Since $I_h^0$ is invariant under the addition of constants, we can replace $v$ by its zero-mean part when applying \eqref{eq:mean-zero-interpolation-H2-new}.
	Quasi-uniformity, \eqref{eq:mean-zero-interpolation-H2-new}, and lifting norm equivalences then imply
	\[
	\left(\sum_{e\in\cE_h}h_e^{-1}
	\|J_{n,e}(I_h^0v)\|_{0,e}^2\right)^{1/2}
	\leq C_\gamma h\|v\|_{3,\gamma}.
	\]
	
	Let $\Pi_e^r$ be the $L^2(e)$ projection onto $P_r(e)$. For $(r,y_h)=(2,v_h)$ or $(1,I_h^0v)$, the corresponding orthogonality, weighted Cauchy--Schwarz inequality, and trace approximation give
	\begin{align*}
		\left|\sum_{e\in\cE_h}
		((\Delta_\gamma w)^e,J_{n,e}(y_h))_e\right|
		&\leq
		\left(\sum_e h_e
		\|(\Delta_\gamma w)^e-\Pi_e^r(\Delta_\gamma w)^e\|_{0,e}^2\right)^{1/2}
		\left(\sum_e h_e^{-1}\|J_{n,e}(y_h)\|_{0,e}^2\right)^{1/2}\\
		&\leq C_\gamma h\|w\|_{3,\gamma}
		\left(\sum_e h_e^{-1}\|J_{n,e}(y_h)\|_{0,e}^2\right)^{1/2}.
	\end{align*}
	Combining this estimate with both jump bounds proves \eqref{eq:edge-consistency-energy-new} and \eqref{eq:edge-consistency-interpolants-new}.
\end{proof}

\begin{lemma}[Right-hand side consistency]
	\label{lem:source-consistency-new}
	For each $v_h\in\mathring V_h$,
	\begin{equation}
		\left|
		l_h(v_h)-l\bigl((\widetilde{\mathcal{A}}_h v_h)^\ell\bigr)
		\right|
		\leq C_\gamma
		h^2\|f\|_{0,\gamma}\|\Delta_hv_h\|_{0,h}.
		\label{eq:source-consistency-companion-new}
	\end{equation}
	For each $v\in\mathring H^3(\gamma)$,
	\begin{equation}
		|l_h(I_h^0v)-l(v)|
		\leq C_\gamma h^2\|f\|_{0,\gamma}\|v\|_{3,\gamma}.
		\label{eq:source-consistency-interpolant-new}
	\end{equation}
\end{lemma}
\begin{proof}
	The discrete zero-mean property and area transformation yield the exact identities
	\begin{align*}
		l_h(v_h)-l\bigl((\widetilde{\mathcal{A}}_h v_h)^\ell\bigr)
		&= (f^e,v_h-\widetilde{\mathcal{A}}_h v_h)_h
		+(f^e(1-\mu_h),\widetilde{\mathcal{A}}_h v_h)_h,\\
		l_h(I_h^0v)-l(v)
		&= (f^e,I_h^0v-v^e)_h+(f^e(1-\mu_h),v^e)_h.
	\end{align*}
	The first identity follows from \eqref{eq:source-basic-new}; the second uses $I_h^0v\in\mathring V_h$. Equations \eqref{eq:geometry-area}, \eqref{eq:companion-approx-new}--\eqref{eq:companion-stability-new}, and \eqref{eq:mean-zero-interpolation-H1-new} prove \eqref{eq:source-consistency-companion-new}--\eqref{eq:source-consistency-interpolant-new}.
\end{proof}

\begin{lemma}[Second-order bilinear consistency for interpolants]
	\label{lem:bilinear-superconsistency-new}
	For $w,v\in\mathring H^3(\gamma)$,
	\begin{equation}
		\left|
		a_h(I_h^0w,I_h^0v)-a(w,v)
		\right|
		\leq C_\gamma h^2\|w\|_{3,\gamma}\|v\|_{3,\gamma}.
		\label{eq:bilinear-superconsistency-new}
	\end{equation}
\end{lemma}
\begin{proof}
	Let $w_I=I_h^0w$, $v_I=I_h^0v$,
	$q_w=\Delta_\gamma w$, $q_v=\Delta_\gamma v$,
	$E_w=\Delta_hw_I-q_w^e$, and $E_v=\Delta_hv_I-q_v^e$.
	By \eqref{eq:mean-zero-interpolation-H2-new},
	\[
	\|E_w\|_{0,h}\leq C_\gamma h\|w\|_{3,\gamma},
	\qquad
	\|E_v\|_{0,h}\leq C_\gamma h\|v\|_{3,\gamma},
	\]
	so $|(E_w,E_v)_h|\leq C_\gamma h^2\|w\|_{3,\gamma}\|v\|_{3,\gamma}$.
	
	For the cross term $(E_w,q_v^e)_h$, cell-wise integration by parts and the identity on closed surfaces
	\[
	(\nabla_\gamma w,\nabla_\gamma q_v)_\gamma
	+(q_w,q_v)_\gamma=0
	\]
	give
	\begin{align*}
		(E_w,q_v^e)_h
		={}&
		-(\nabla_h(w_I-w^e),\nabla_hq_v^e)_h
		+\left[
		(\nabla_\gamma w,\nabla_\gamma q_v)_\gamma
		-(\nabla_hw^e,\nabla_hq_v^e)_h
		\right]\\
		&+\sum_{e\in\cE_h}
		(q_v^e,J_{n,e}(w_I))_e
		+\left[(q_w,q_v)_\gamma-(q_w^e,q_v^e)_h\right].
	\end{align*}
	The first term is controlled by \eqref{eq:mean-zero-interpolation-H1-new}, the second by \eqref{eq:gradient-transport}, the third by \eqref{eq:edge-consistency-interpolants-new}, and the fourth by \eqref{eq:geometry-area}. Consequently,
	\[
	|(E_w,q_v^e)_h|
	\leq C_\gamma h^2\|w\|_{3,\gamma}\|v\|_{3,\gamma}.
	\]
	The identical argument swapping $w$ and $v$ controls $(q_w^e,E_v)_h$.
	Finally,
	\begin{align*}
		a_h(w_I,v_I)-a(w,v)
		=(E_w,E_v)_h+(E_w,q_v^e)_h+(q_w^e,E_v)_h
		+\left[(q_w^e,q_v^e)_h-(q_w,q_v)_\gamma\right].
	\end{align*}
	By \eqref{eq:geometry-area} and the definitions of $q_w$ and $q_v$, the last term is bounded by $C_\gamma h^2\|w\|_{3,\gamma}\|v\|_{3,\gamma}$; previous bounds prove \eqref{eq:bilinear-superconsistency-new}.
\end{proof}

\subsection{Energy-norm and broken $H^1$ error estimates}

\begin{theorem}[Energy-norm error estimate]
	\label{thm:energy-norm-error-new}
	Let $f\in\mathring L^2(\gamma)$, and let $u\in\mathring H^3(\gamma)$ and $u_h\in\mathring V_h$ be the solutions of \eqref{eq:continuous-problem-new} and \eqref{eq:discrete-problem-new}, respectively. Then
	\begin{align}
		\|\Delta_h(I_h^0u-u_h)\|_{0,h}
		&\leq C_\gamma
		h\bigl(\|u\|_{3,\gamma}+\|f\|_{0,\gamma}\bigr)
		\leq C_\gamma h\|f\|_{0,\gamma},
		\label{eq:energy-interpolated-error-new}\\
		\|(\Delta_\gamma u)^e-\Delta_hu_h\|_{0,h}
		+\|D_h^2(u^e-u_h)\|_{0,h}
		&\leq C_\gamma
		h\bigl(\|u\|_{3,\gamma}+\|f\|_{0,\gamma}\bigr)
		\leq C_\gamma h\|f\|_{0,\gamma}.
		\label{eq:energy-physical-error-new}
	\end{align}
\end{theorem}
\begin{proof}
	Let $u_I=I_h^0u$ and $q=\Delta_\gamma u$.
	For any $v_h\in\mathring V_h$, let $z_h=\widetilde{\mathcal{A}}_h v_h$. Since $q\in H^1(\gamma)$ and $\Delta_\gamma q=f$, cell-wise integration by parts and \eqref{eq:surface-green} give
	\begin{align*}
		a_h(u_I,v_h)-l_h(v_h)
		={}&(\Delta_hu_I-q^e,\Delta_hv_h)_h
		+\sum_{e\in\cE_h}(q^e,J_{n,e}(v_h))_e\\
		&-(\nabla_hq^e,\nabla_h(v_h-z_h))_h\\
		&+\left[(\nabla_\gamma q,\nabla_\gamma z_h^\ell)_\gamma
		-(\nabla_hq^e,\nabla_hz_h)_h\right]\\
		&+l(z_h^\ell)-l_h(v_h).
	\end{align*}
	The first three terms are controlled by \eqref{eq:mean-zero-interpolation-H2-new}, \eqref{eq:edge-consistency-energy-new}, and \eqref{eq:companion-approx-new}, respectively; the fourth by \eqref{eq:companion-stability-new}, \eqref{eq:gradient-transport}, and \eqref{eq:h1-seminorm-equivalence}; the fifth by \eqref{eq:source-consistency-companion-new}. Thus,
	\begin{equation}
		|a_h(u_I,v_h)-l_h(v_h)|
		\leq C_\gamma
		h\bigl(\|u\|_{3,\gamma}+\|f\|_{0,\gamma}\bigr)
		\|\Delta_hv_h\|_{0,h}.
		\label{eq:energy-residual-bound-new}
	\end{equation}
	
	Taking $v_h=e_h:=u_I-u_h$ in \eqref{eq:energy-residual-bound-new} and using the discrete problem yields \eqref{eq:energy-interpolated-error-new}. Equation \eqref{eq:energy-physical-error-new} follows from \eqref{eq:mean-zero-interpolation-H2-new}, the triangle inequality, and Theorem~\ref{thm:discrete-miranda-talenti-new}.
	The final inequality is derived from \eqref{eq:continuous-regularity-new}.
\end{proof}

\begin{theorem}[Broken $H^1$ error estimate]
	\label{thm:broken-H1-error-new}
	Under the assumptions of Theorem~\ref{thm:energy-norm-error-new},
	\begin{equation}
		\|u^e-u_h\|_{1,h}
		\leq C_\gamma h^2
		\bigl(\|u\|_{3,\gamma}+\|f\|_{0,\gamma}\bigr)
		\leq C_\gamma h^2\|f\|_{0,\gamma}.
		\label{eq:broken-H1-error-new}
	\end{equation}
\end{theorem}
\begin{proof}
	Let $u_I=I_h^0u$, $e_h=u_I-u_h$, and $z_h=\widetilde{\mathcal{A}}_h e_h$. The nonconformity bounds and \eqref{eq:energy-interpolated-error-new} give
	\begin{align}
		\|e_h-z_h\|_{0,h}
		&\leq C_\gamma h^3
		\bigl(\|u\|_{3,\gamma}+\|f\|_{0,\gamma}\bigr),
		\label{eq:dual-companion-L2-new}\\
		|e_h-z_h|_{1,h}
		&\leq C_\gamma h^2
		\bigl(\|u\|_{3,\gamma}+\|f\|_{0,\gamma}\bigr).
		\label{eq:dual-companion-H1-new}
	\end{align}
	Moreover, $z_h\in H^1(\Gamma_h)$ and has zero mean. Define $g=-\Delta_\gamma z_h^\ell$ in the distributional sense by:
	\[
	\langle g,\psi\rangle
	=(\nabla_\gamma z_h^\ell,\nabla_\gamma\psi)_\gamma,
	\qquad \psi\in H^1(\gamma).
	\]
	Since $\langle g,1\rangle=0$, we have $g\in\mathring H^{-1}(\gamma)$, where $\mathring H^{-1}(\gamma):=\{\xi\in H^{-1}(\gamma):\langle\xi,1\rangle=0\}$, and $\|\xi\|_{-1,\gamma}
	:=\sup_{0\neq\psi\in\mathring H^1(\gamma)}
	\frac{|\langle\xi,\psi\rangle|}{\|\psi\|_{1,\gamma}}$.
	Let $\phi\in\mathring H^3(\gamma)$ solve $\Delta_\gamma^2\phi=g$.
	The zero-mean $H^{-1}\to H^3$ regularity estimate \cite[section~4.3, eq.~(4.12)]{WuZhou2025} yields
	\begin{equation}
		\|\phi\|_{3,\gamma}
		\leq C_\gamma\|g\|_{-1,\gamma}
		\leq C_\gamma|z_h^\ell|_{1,\gamma}.
		\label{eq:dual-regularity-new}
	\end{equation}
	
	Let $q_\phi=\Delta_\gamma\phi$. Adding and subtracting $(\nabla_hq_\phi^e,\nabla_hz_h)_h$ and integrating elementwise by parts give
	\begin{align*}
		|z_h^\ell|_{1,\gamma}^2
		={}&\bigl[-(\nabla_\gamma q_\phi,\nabla_\gamma z_h^\ell)_\gamma
		+(\nabla_hq_\phi^e,\nabla_hz_h)_h\bigr]
		-(\nabla_hq_\phi^e,\nabla_h(z_h-e_h))_h\\
		&+(q_\phi^e,\Delta_he_h)_h
		-\sum_{e\in\cE_h}(q_\phi^e,J_{n,e}(e_h))_e.
	\end{align*}
	The first term is controlled by \eqref{eq:gradient-transport}, the second by \eqref{eq:h1-seminorm-equivalence} and \eqref{eq:dual-companion-H1-new}.
	For $\phi_I=I_h^0\phi$, the discrete and continuous problems give
	\[
	(q_\phi^e,\Delta_he_h)_h
	=(q_\phi^e-\Delta_h\phi_I,\Delta_he_h)_h
	+\bigl[a_h(u_I,\phi_I)-a(u,\phi)\bigr]
	+\bigl[l(\phi)-l_h(\phi_I)\bigr].
	\]
	These three terms are controlled by \eqref{eq:mean-zero-interpolation-H2-new} and \eqref{eq:energy-interpolated-error-new}, \eqref{eq:bilinear-superconsistency-new}, and \eqref{eq:source-consistency-interpolant-new}, respectively. The edge jump term is controlled by \eqref{eq:edge-consistency-energy-new} and \eqref{eq:energy-interpolated-error-new}. Hence, by \eqref{eq:dual-regularity-new},
	\[
	\begin{aligned}
		|z_h^\ell|_{1,\gamma}^2
		&\leq C_\gamma h^2\|\phi\|_{3,\gamma}
		\left[|z_h^\ell|_{1,\gamma}
		+\|u\|_{3,\gamma}+\|f\|_{0,\gamma}\right]\\
		&\leq C_\gamma h^2\left[
		|z_h^\ell|_{1,\gamma}^2
		+|z_h^\ell|_{1,\gamma}
		\bigl(\|u\|_{3,\gamma}+\|f\|_{0,\gamma}\bigr)\right].
	\end{aligned}
	\]
	For sufficiently small $h$, absorption and cancellation (with the zero seminorm case being trivial) give
	\[
	|z_h^\ell|_{1,\gamma}
	\leq C_\gamma h^2
	\bigl(\|u\|_{3,\gamma}+\|f\|_{0,\gamma}\bigr).
	\]
	Equations \eqref{eq:h1-seminorm-equivalence} and \eqref{eq:dual-companion-H1-new} transfer this bound to $|e_h|_{1,h}$; the zero-mean property of $z_h$, \eqref{eq:uniform-poincare-discrete-surface}, and \eqref{eq:dual-companion-L2-new} provide the corresponding $L^2$ bound.
	Combining these estimates with \eqref{eq:mean-zero-interpolation-H1-new} proves \eqref{eq:broken-H1-error-new}. The final inequality is derived from \eqref{eq:continuous-regularity-new}.
\end{proof}

\begin{remark}[Energy norm and geometric consistency]
	The natural energy norm of \eqref{eq:discrete-problem-new} is $\|\Delta_hv_h\|_{0,h}$. The estimate \eqref{eq:discrete-MT-new} provides broken $H^2$ control, so no jump or stabilization seminorm enters this norm. In the variational consistency analysis, the geometric perturbations other than the direct Laplacian transport error \eqref{eq:laplace-geometric-error} are of order $h^2$. They arise from the area and tangential-gradient transport estimates \eqref{eq:geometry-area} and \eqref{eq:gradient-transport}, as reflected in the second-order consistency bound \eqref{eq:bilinear-superconsistency-new}. The proved first-order energy bound contains both the finite element approximation error from \eqref{eq:local-interpolation-new} based on quadratic reproduction and the Laplacian transport error from \eqref{eq:laplace-geometric-error}. Equation \eqref{eq:broken-H1-error-new} yields second-order $L^2$ convergence via the full broken $H^1$ estimate. Higher-order surface approximations can reduce geometric errors \cite{Demlow2009}, but for the current space, third-order $L^2$ estimates would require additional approximation or dual consistency arguments beyond those used here.
\end{remark}

\section{Numerical Experiments}
\label{sec:numerical-experiments}

The first three benchmark examples test the predicted first-order energy-norm convergence and second-order broken $H^1$ and $L^2$ convergence using the geometries and exact solutions in \cite{Dziuk1988,OlshanskiiReuskenGrande2009,WuZhou2025}. Sections~\ref{subsec:numerics-mt}--\ref{subsec:numerics-angle} examine the discrete Miranda--Talenti constants and angle defect scaling on the same meshes and their perturbations. For the convergence tests, element, load, and error integrals are evaluated in MATLAB using Dunavant's 27-point quadrature rule \cite{Dunavant1985}. The sphere and torus meshes are generated independently by Gmsh \cite{GeuzaineRemacle2009}. The implicit surface mesh construction is described in Section~\ref{subsec:numerics-implicit}. We report the number of degrees of freedom (DoFs) along with the following mesh sizes:
\[
h_{\max}:=h=\max_{K\in\mathcal T_h}\operatorname{diam}(K),
\qquad
h_{\mathrm{mean}}
:=\frac{1}{3|\mathcal T_h|}
\sum_{K\in\mathcal T_h}\sum_{e\subset\partial K}|e|,
\]
where paired entries $r_{\max}/r_{\mathrm{mean}}$ give the experimental convergence rates computed using $h_{\max}$ and $h_{\mathrm{mean}}$, respectively. Here, $h_{\mathrm{mean}}$ serves merely as a diagnostic indicator for independently generated meshes, whereas $h_{\max}$ is the analytical mesh parameter. In the convergence tests, the zero-mean condition is enforced by a Lagrange multiplier, whereas in spectral diagnostics it is imposed by an equivalent constant quotient space formulation. For the first two examples, we compute
\[
E_0:=\|u^e-u_h\|_{0,h},\qquad
E_1:=\|u^e-u_h\|_{1,h},\qquad
E_\Delta:=\|(\Delta_\gamma u)^e-\Delta_hu_h\|_{0,h}.
\]

\subsection{Example 1: Unit sphere}
\label{subsec:numerics-sphere}

On the unit sphere $\gamma=\{(x,y,z)\in\mathbb R^3:x^2+y^2+z^2=1\}$, we take the cubic spherical harmonic $u=3x^2y-y^3$, for which $f=\Delta_\gamma^2u=144u$. Figure~\ref{fig:b4-sphere-solution} displays the numerical solution, and Table~\ref{tab:b4-sphere-errors} shows that on refined meshes, both $E_0$ and $E_1$ achieve second-order convergence and $E_\Delta$ achieves first-order convergence with respect to both mesh-size measures, in agreement with the theoretical analysis.

\begin{figure}[htbp]
	\centering
	\includegraphics[
	width=0.48\textwidth,
	alt={Triangulated surface plot of the numerical solution on the unit
		sphere. Six alternating positive and negative lobes surround the equator,
		separated by nodal curves, with values ranging approximately from -1 to 1.}]
	{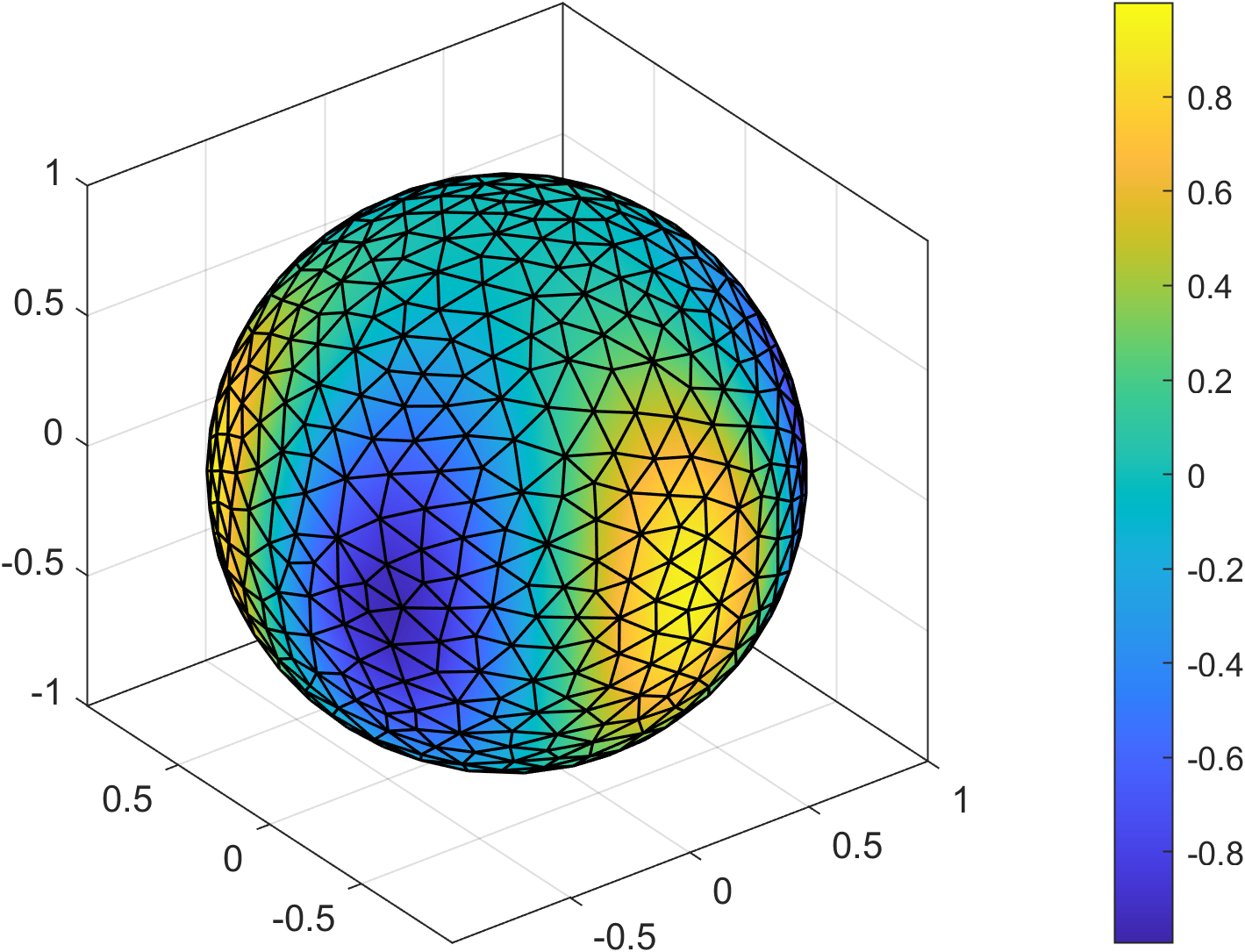}
	\caption{Numerical approximation of $u=3x^2y-y^3$ on the unit sphere with $6632$ degrees of freedom.}
	\label{fig:b4-sphere-solution}
\end{figure}

\begin{table}[htbp]
	\centering
	\caption{Errors and experimental convergence rates for the unit sphere example.}
	\label{tab:b4-sphere-errors}
	\small
	\setlength{\tabcolsep}{4pt}
	\resizebox{\textwidth}{!}{%
		\begin{tabular}{rrrrrrrrr}
			\toprule
			\multicolumn{1}{c}{DoFs} & \multicolumn{1}{c}{$h_{\max}$}
			& \multicolumn{1}{c}{$h_{\mathrm{mean}}$}
			& \multicolumn{1}{c}{$E_0$} & \multicolumn{1}{c}{$r_{\max}/r_{\mathrm{mean}}$}
			& \multicolumn{1}{c}{$E_1$} & \multicolumn{1}{c}{$r_{\max}/r_{\mathrm{mean}}$}
			& \multicolumn{1}{c}{$E_\Delta$} & \multicolumn{1}{c}{$r_{\max}/r_{\mathrm{mean}}$} \\
			\midrule
			1572 & $3.976\mathrm e{-01}$ & $3.043\mathrm e{-01}$
			& $2.1878\mathrm e{-01}$ & --/--
			& $8.0849\mathrm e{-01}$ & --/--
			& $6.0269\mathrm e{+00}$  & --/-- \\
			6632 & $2.130\mathrm e{-01}$ & $1.492\mathrm e{-01}$
			& $5.1633\mathrm e{-02}$ & 2.313/2.026
			& $1.9529\mathrm e{-01}$ & 2.276/1.993
			& $2.9319\mathrm e{+00}$  & 1.154/1.011 \\
			26432 & $1.110\mathrm e{-01}$ & $7.482\mathrm e{-02}$
			& $1.3545\mathrm e{-02}$ & 2.054/1.939
			& $5.1620\mathrm e{-02}$ & 2.042/1.928
			& $1.5136\mathrm e{+00}$  & 1.015/0.958 \\
			105742 & $5.590\mathrm e{-02}$ & $3.743\mathrm e{-02}$
			& $3.3488\mathrm e{-03}$ & 2.037/2.017
			& $1.2833\mathrm e{-02}$ & 2.028/2.009
			& $7.5236\mathrm e{-01}$ & 1.019/1.009 \\
			420952 & $2.799\mathrm e{-02}$ & $1.876\mathrm e{-02}$
			& $8.3374\mathrm e{-04}$ & 2.011/2.013
			& $3.1921\mathrm e{-03}$ & 2.012/2.014
			& $3.7497\mathrm e{-01}$ & 1.007/1.008 \\
			\bottomrule
	\end{tabular}}
\end{table}

\subsection{Example 2: Torus}
\label{subsec:numerics-torus}

Consider a torus with major radius $R=1$ and minor radius $r=0.6$, parameterized by
\[
x=(R+r\cos\theta)\cos\varphi,
\qquad
y=(R+r\cos\theta)\sin\varphi,
\qquad
z=r\sin\theta,
\]
where $0\leq\theta,\varphi<2\pi$. The exact solution is $u(\theta,\varphi)=\sin(3\varphi)\cos(3\theta+\varphi)$, with both $\Delta_\gamma u$ and $f=\Delta_\gamma^2u$ computed analytically. Figure~\ref{fig:b4-torus-solution} illustrates the more oscillatory numerical solution. Table~\ref{tab:b4-torus-errors} shows preasymptotic behavior on the coarser meshes. On the finer meshes, the results are broadly consistent with second-order convergence for $E_0$ and $E_1$ and first-order convergence for $E_\Delta$. The final fluctuation in $E_0$ may be related to local differences between independently generated meshes.

\begin{figure}[htbp]
	\centering
	\includegraphics[
	width=0.48\textwidth,
	alt={Triangulated surface plot of the numerical solution on the torus.
		Alternating positive and negative patches extend around both angular
		directions, with values ranging approximately from -1 to 1.}]
	{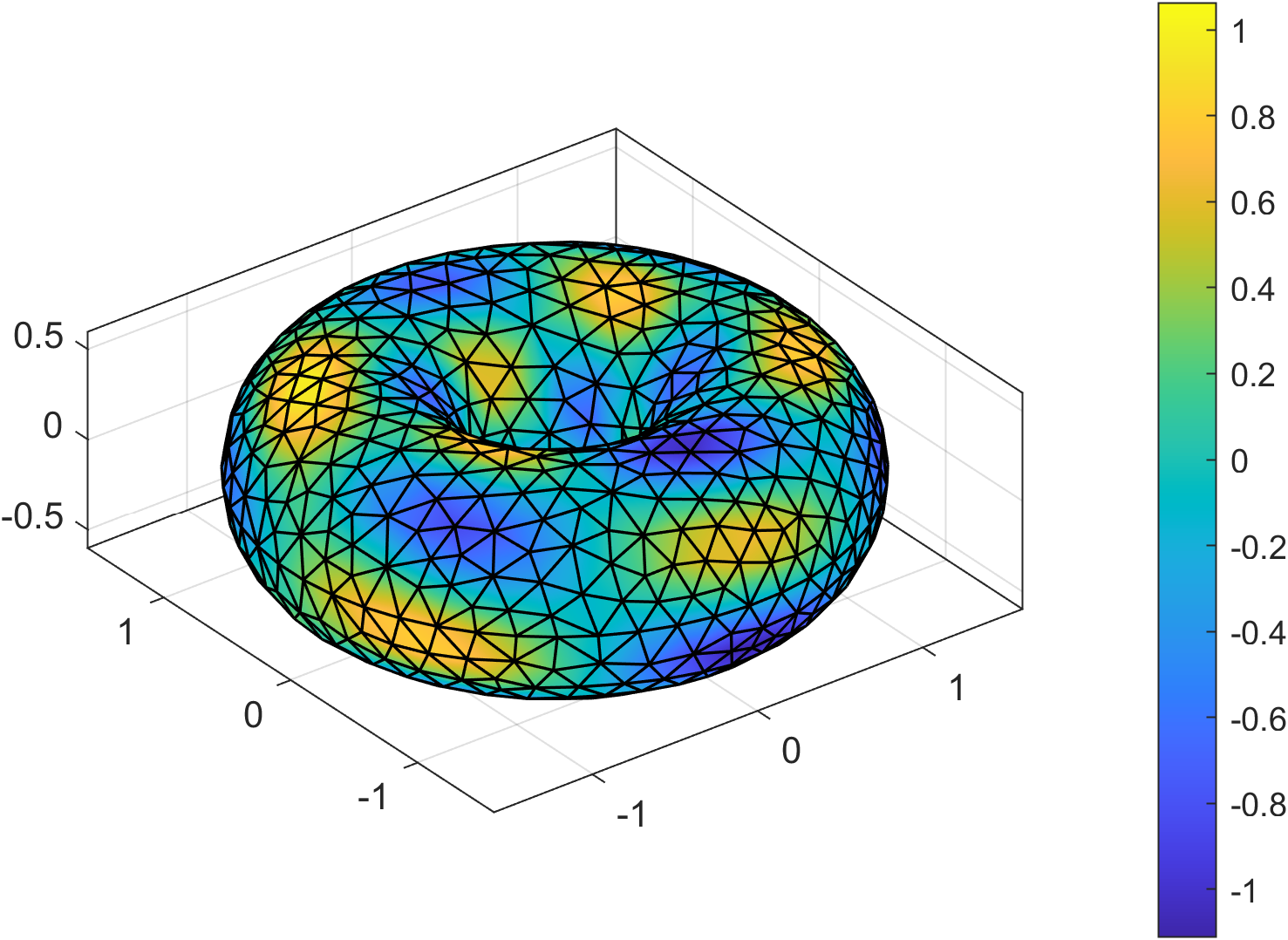}
	\caption{Numerical approximation of $u=\sin(3\varphi)\cos(3\theta+\varphi)$ on the torus with $7710$ degrees of freedom.}
	\label{fig:b4-torus-solution}
\end{figure}

\begin{table}[htbp]
	\centering
	\caption{Errors and experimental convergence rates for the torus example.}
	\label{tab:b4-torus-errors}
	\small
	\setlength{\tabcolsep}{4pt}
	\resizebox{\textwidth}{!}{%
		\begin{tabular}{rrrrrrrrr}
			\toprule
			\multicolumn{1}{c}{DoFs} & \multicolumn{1}{c}{$h_{\max}$}
			& \multicolumn{1}{c}{$h_{\mathrm{mean}}$}
			& \multicolumn{1}{c}{$E_0$} & \multicolumn{1}{c}{$r_{\max}/r_{\mathrm{mean}}$}
			& \multicolumn{1}{c}{$E_1$} & \multicolumn{1}{c}{$r_{\max}/r_{\mathrm{mean}}$}
			& \multicolumn{1}{c}{$E_\Delta$} & \multicolumn{1}{c}{$r_{\max}/r_{\mathrm{mean}}$} \\
			\midrule
			1970 & $5.379\mathrm e{-01}$ & $3.763\mathrm e{-01}$
			& $4.0780\mathrm e{+00}$  & --/--
			& $8.8585\mathrm e{+00}$  & --/--
			& $6.6924\mathrm e{+01}$  & --/-- \\
			7710 & $2.699\mathrm e{-01}$ & $1.904\mathrm e{-01}$
			& $6.4589\mathrm e{-01}$ & 2.672/2.705
			& $2.9550\mathrm e{+00}$  & 1.592/1.612
			& $4.4990\mathrm e{+01}$  & 0.576/0.583 \\
			31070 & $1.387\mathrm e{-01}$ & $9.479\mathrm e{-02}$
			& $1.6578\mathrm e{-01}$ & 2.042/1.950
			& $9.1423\mathrm e{-01}$ & 1.762/1.682
			& $2.5152\mathrm e{+01}$  & 0.873/0.834 \\
			124870 & $6.945\mathrm e{-02}$ & $4.730\mathrm e{-02}$
			& $4.2158\mathrm e{-02}$ & 1.980/1.970
			& $2.3224\mathrm e{-01}$ & 1.982/1.971
			& $1.2753\mathrm e{+01}$  & 0.982/0.977 \\
			496630 & $3.500\mathrm e{-02}$ & $2.372\mathrm e{-02}$
			& $8.0173\mathrm e{-03}$ & 2.422/2.405
			& $5.8739\mathrm e{-02}$ & 2.006/1.992
			& $6.4586\mathrm e{+00}$  & 0.993/0.986 \\
			\bottomrule
	\end{tabular}}
\end{table}

\subsection{Example 3: General implicit surface}
\label{subsec:numerics-implicit}

The third example is set on the zero level set $\gamma=\{(x,y,z)\in\mathbb R^3:\phi(x,y,z)=0\}$, where $\phi(x,y,z)=(x-z^2)^2+y^2+z^2-1$. The exact solution is $u=y$, and $f=\Delta_\gamma^2u$; analytical expressions for $\Delta_\gamma u$ and $f$ are generated beforehand via symbolic computation.

At each level, isosurfaces from independent Cartesian grids are projected onto $\gamma$, processed via tangential smoothing with reprojection, remeshed by Gmsh, and projected once more. For a point $X$, its closest point $Y=p(X)$ is computed by solving the KKT system $Y-X+\lambda\nabla\phi(Y)=0, \phi(Y)=0$ via damped Newton iteration, where $\lambda$ is a Lagrange multiplier; see \cite[section~4.2]{DemlowDziuk2007}. The same projection is applied to the quadrature nodes. Once the mesh is constructed, the assembly of the stiffness matrix proceeds exactly as on a planar triangulation. The minimum angles of the five final meshes range within $[19.57^\circ,31.67^\circ]$, and nodal level-set residuals are below $5\times10^{-13}$, indicating satisfactory mesh quality and nodal projection accuracy.

Since $\phi$ is not a signed distance function, the Hessian of the associated signed distance is not readily available explicitly. Therefore, following \cite[section~5.3]{WuZhou2025}, we report $E_0$, $E_\Delta$, and
\begin{align*}
	E_1^\star
	&:=\|P_h(\nabla_\gamma u)^e-\nabla_hu_h\|_{0,h},\\
	\left|E_1^\star-|u^e-u_h|_{1,h}\right|
	&\leq\|P_h(\nabla_\gamma u)^e-\nabla_hu^e\|_{0,h}
	\lesssim h^2\mathcal K_0^2\|u\|_{1,\gamma},
\end{align*}
where $P_h$ is the element tangential projection. The identity $\nabla_hu^e=P_h(P-dH)(\nabla_\gamma u)^e$ shows that $E_0$ and $E_1^\star$ together test the second-order estimate in the full broken $H^1$ norm. Figure~\ref{fig:b4-implicit-solution} depicts the numerical solution, while Table~\ref{tab:b4-implicit-errors} reports overall second-order convergence for $E_0$ and $E_1^\star$ and first-order convergence for $E_\Delta$. Relatively stable $r_{\mathrm{mean}}$ values indicate that fluctuations in $r_{\max}$ may be related to a few exceptionally large elements in the independent non-nested meshes; the overall error decay remains consistent with the analysis.

\begin{figure}[htbp]
	\centering
	\includegraphics[
	width=0.48\textwidth,
	alt={Triangulated surface plot of the numerical solution on the asymmetric
		implicit surface. The values change from negative on the y-negative side to
		positive on the y-positive side, with a zero-level band near y equals zero
		and a range of approximately -1 to 1.}]
	{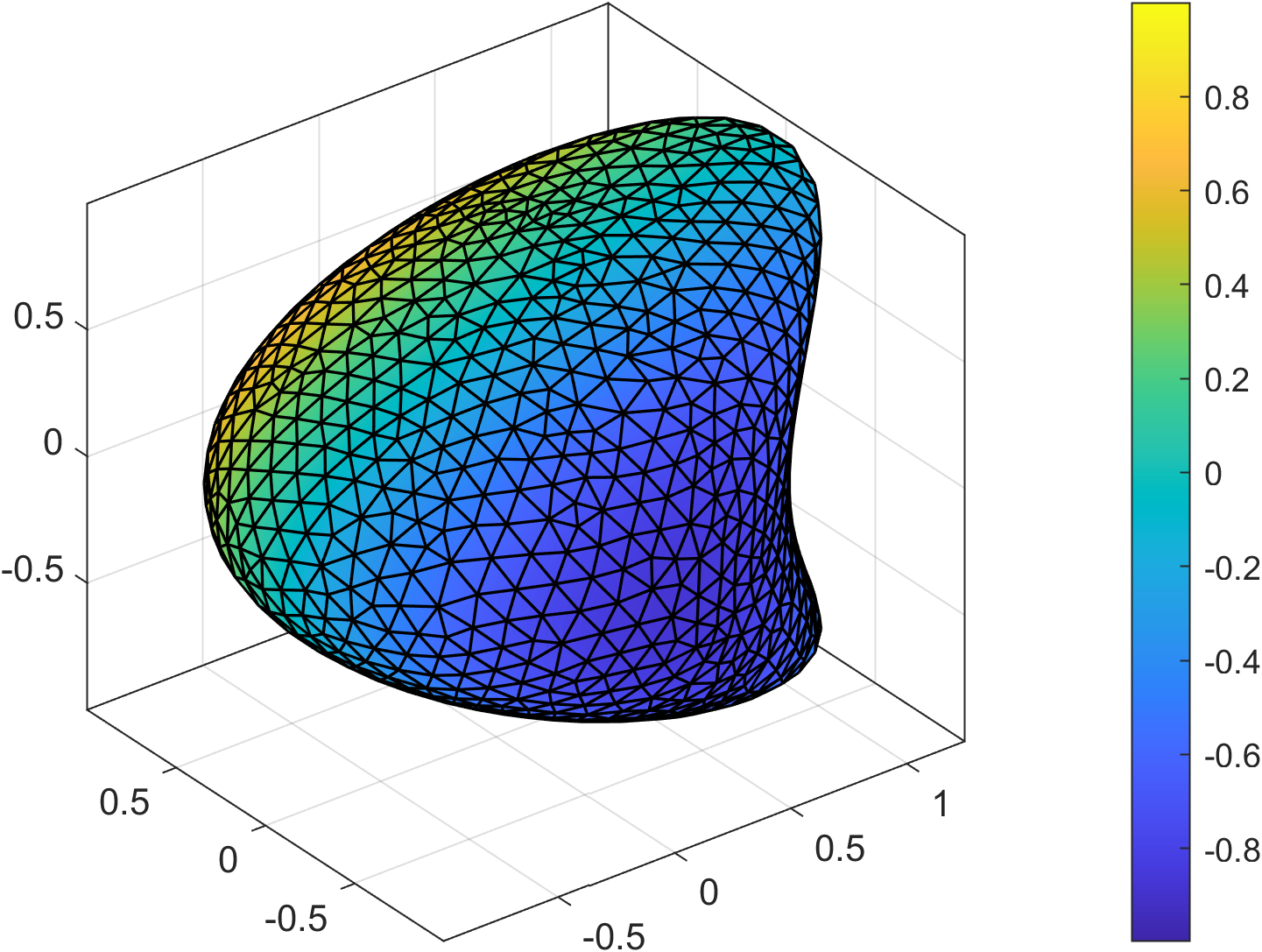}
	\caption{Numerical approximation of $u=y$ on the implicit surface with $11982$ degrees of freedom.}
	\label{fig:b4-implicit-solution}
\end{figure}

\begin{table}[htbp]
	\centering
	\caption{Errors and experimental convergence rates for the implicit surface example.}
	\label{tab:b4-implicit-errors}
	\small
	\setlength{\tabcolsep}{4pt}
	\resizebox{\textwidth}{!}{%
		\begin{tabular}{rrrrrrrrr}
			\toprule
			\multicolumn{1}{c}{DoFs} & \multicolumn{1}{c}{$h_{\max}$}
			& \multicolumn{1}{c}{$h_{\mathrm{mean}}$}
			& \multicolumn{1}{c}{$E_0$} & \multicolumn{1}{c}{$r_{\max}/r_{\mathrm{mean}}$}
			& \multicolumn{1}{c}{$E_1^\star$} & \multicolumn{1}{c}{$r_{\max}/r_{\mathrm{mean}}$}
			& \multicolumn{1}{c}{$E_\Delta$} & \multicolumn{1}{c}{$r_{\max}/r_{\mathrm{mean}}$} \\
			\midrule
			2972 & $3.415\mathrm e{-01}$ & $2.302\mathrm e{-01}$
			& $1.2111\mathrm e{+00}$  & --/--
			& $1.7894\mathrm e{+00}$  & --/--
			& $3.3958\mathrm e{+00}$  & --/-- \\
			11982 & $1.855\mathrm e{-01}$ & $1.149\mathrm e{-01}$
			& $2.9888\mathrm e{-01}$ & 2.294/2.015
			& $4.4110\mathrm e{-01}$ & 2.296/2.017
			& $1.3468\mathrm e{+00}$  & 1.516/1.332 \\
			46172 & $8.249\mathrm e{-02}$ & $5.853\mathrm e{-02}$
			& $6.3365\mathrm e{-02}$ & 1.914/2.298
			& $9.6578\mathrm e{-02}$ & 1.874/2.251
			& $6.5433\mathrm e{-01}$ & 0.891/1.070 \\
			183392 & $4.902\mathrm e{-02}$ & $2.935\mathrm e{-02}$
			& $1.6248\mathrm e{-02}$ & 2.615/1.971
			& $2.4782\mathrm e{-02}$ & 2.614/1.970
			& $3.4480\mathrm e{-01}$ & 1.231/0.928 \\
			746002 & $2.276\mathrm e{-02}$ & $1.460\mathrm e{-02}$
			& $4.0736\mathrm e{-03}$ & 1.803/1.981
			& $6.2175\mathrm e{-03}$ & 1.802/1.980
			& $1.7083\mathrm e{-01}$ & 0.915/1.006 \\
			\bottomrule
	\end{tabular}}
\end{table}

% Diagnostics from saved runs 20261012_134740 (MT) and 20261012_154856 (angles).
% Tables include only qualified MT eigenpairs; geometry uses all 60 meshes.

\FloatBarrier
\subsection{Discrete Miranda--Talenti constants}
\label{subsec:numerics-mt}

We reuse all five meshes from each of the first three examples and test each original mesh along with three perturbed meshes obtained by sampling vertex displacements within a tangential disk of radius $0.05h_z$, with $h_z$ evaluated on the original mesh, and projecting them onto the smooth surface. The connectivity and degrees of freedom remain unchanged. All $60$ meshes pass the geometric check, with minimum angles exceeding $18^\circ$. We measure the optimal constant as
\begin{equation*}
	C_{\mathrm{MT},h}
	:=\sup_{0\ne v_h\in\mathring V_h}
	\frac{\|D_h^2v_h\|_{0,h}}{\|\Delta_hv_h\|_{0,h}}.
\end{equation*}
Proposition~\ref{prop:global-harmonic-kernel} ensures that $C_{\mathrm{MT},h}$ is finite on each fixed mesh; this experiment examines its behavior under refinement in relation to \eqref{eq:laplace-hessian-equivalence-new}.

We compute $C_{\mathrm{MT},h}^{2}$ as the largest generalized eigenvalue of the Hessian and Laplacian forms on the constant quotient space, equivalently on $\mathring V_h$, without regularization. Both forms are integrated using a positive 49-point, degree-12 Gauss--Duffy quadrature rule. A computation is accepted if the relative eigenpair residuals are at most~$10^{-8}$~in both the original and Cholesky-transformed systems, the relative mean residuals are below $10^{-10}$, and the relative cell-wise Rayleigh quotient differences are below $10^{-7}$. Rechecking selected accepted results with degree-14 quadrature and stricter solver tolerances yields relative changes of less than $5.6\times10^{-10}$.

Table~\ref{tab:numerics-mt} and Figure~\ref{fig:numerics-mt} report the first four mesh sizes for each surface, for which all $48$ computations are accepted. The table lists the original-mesh values and the minimum-to-maximum ranges over the three perturbed meshes; $\theta_{\min}$ denotes the minimum triangle angle of the original mesh. In the figure, dashed circles and lines show the original-mesh values, while bars show the perturbation ranges at the corresponding original mesh size $h_{\max}$. The computed constants lie between $1.15326$ and $1.28695$, and perturbations change the original-mesh values by less than $0.21\%$. These results are consistent with the discrete estimate \eqref{eq:laplace-hessian-equivalence-new} over the tested refinement range.
On the finest meshes, seven computations fail the accuracy checks. The finest-mesh results for all three surfaces are omitted from Table~\ref{tab:numerics-mt} and Figure~\ref{fig:numerics-mt}.

\begin{table}[htbp]
	\centering
	\caption{Discrete Miranda--Talenti constants.}
	\label{tab:numerics-mt}
	\small
	\setlength{\tabcolsep}{6pt}
	\begin{tabular}{lrrrrl}
		\toprule
		\multicolumn{1}{c}{Surface} & \multicolumn{1}{c}{DoFs}
		& \multicolumn{1}{c}{$h_{\max}$} & \multicolumn{1}{c}{$\theta_{\min}$}
		& \multicolumn{1}{c}{$C_{\mathrm{MT},h}$} & \multicolumn{1}{c}{Perturbation range} \\
		\midrule
		Sphere & 1572 & 0.3976 & $30.1^\circ$ & 1.15377 & [1.15326, 1.15494] \\
		& 6632 & 0.2130 & $30.6^\circ$ & 1.16611 & [1.16486, 1.16683] \\
		& 26432 & 0.1110 & $30.4^\circ$ & 1.16838 & [1.16730, 1.16815] \\
		& 105742 & 0.0559 & $29.0^\circ$ & 1.17494 & [1.17445, 1.17559] \\
		\midrule
		Torus & 1970 & 0.5379 & $32.4^\circ$ & 1.19390 & [1.19186, 1.19539] \\
		& 7710 & 0.2699 & $26.5^\circ$ & 1.24497 & [1.24469, 1.24504] \\
		& 31070 & 0.1387 & $26.8^\circ$ & 1.27734 & [1.27724, 1.27740] \\
		& 124870 & 0.0695 & $28.1^\circ$ & 1.28695 & [1.28693, 1.28694] \\
		\midrule
		Implicit surface & 2972 & 0.3415 & $28.2^\circ$ & 1.15978 & [1.15889, 1.15992] \\
		& 11982 & 0.1855 & $31.7^\circ$ & 1.16462 & [1.16222, 1.16370] \\
		& 46172 & 0.0825 & $20.2^\circ$ & 1.16667 & [1.16506, 1.16587] \\
		& 183392 & 0.0490 & $24.3^\circ$ & 1.16644 & [1.16495, 1.16678] \\
		\bottomrule
	\end{tabular}
\end{table}

\begin{figure}[htbp]
	\centering
	\includegraphics[width=0.66\textwidth,
	alt={Computed discrete Miranda--Talenti constants on the first four mesh sizes
		of the sphere, torus, and implicit surface. Dashed circles and lines show
		original meshes; solid bars show the range over three accepted perturbations.
		All displayed groups satisfy the accuracy checks, and refinement proceeds
		from left to right.}]
	{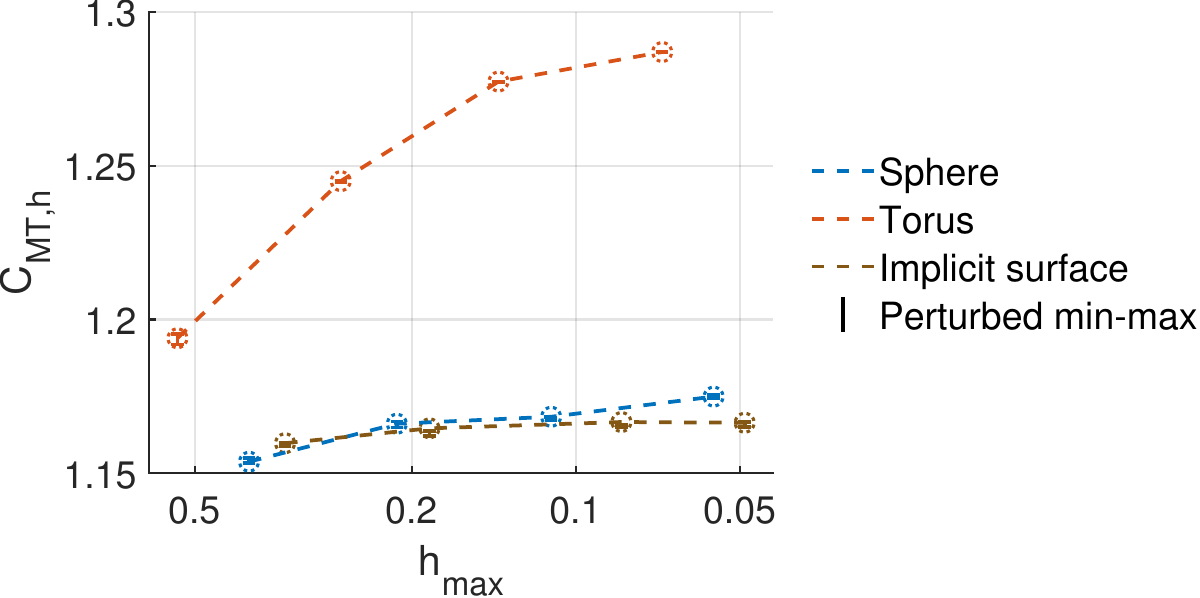}
	\caption{Discrete Miranda--Talenti constants under mesh refinement.}
	\label{fig:numerics-mt}
\end{figure}

\FloatBarrier
\subsection{Angle defect scaling}
\label{subsec:numerics-angle}

On all $60$ meshes from Section~\ref{subsec:numerics-mt}, regardless of eigenpair accuracy, we compute
\begin{equation*}
	\Lambda_{\kappa,h}:=\max_{z\in\cV_h}\frac{|\kappa_z|}{h_z^2},
	\qquad
	K_{\kappa,h}:=\max_{z\in\cV_h}|\kappa_z|,
\end{equation*}
where $\kappa_z$ comes from \eqref{eq:angle-defect-definition-new} and $h_z$, as defined in Section~\ref{sec:preliminaries}, is evaluated on each mesh. We use $\Lambda_{\kappa,h}$ to examine the discrete curvature scale appearing in \eqref{eq:abstract-angle-defect-bound}, rather than to estimate the pointwise Gaussian curvature.

Table~\ref{tab:numerics-angle} lists the original-mesh diagnostics and the perturbation ranges of $\Lambda_{\kappa,h}$. Figure~\ref{fig:numerics-angle} uses the same markers and range bars as Figure~\ref{fig:numerics-mt}, plotted against the original mesh size $h_{\max}$. Including perturbations, the normalized defects on the sphere, torus, and implicit surface remain below $0.81$, $4.0$, and $20.1$, respectively (Figure~\ref{fig:numerics-angle}(a)). The original sphere and torus values remain within relatively narrow ranges, whereas the original implicit surface values rise from $8.10$ to $20.03$ without reaching a plateau. The absolute defects decrease across all three surfaces (Figure~\ref{fig:numerics-angle}(b)), consistent with Lemma~\ref{lem:geometric-admissibility-new} over the measured mesh sizes. The finite-mesh tests in Sections~\ref{subsec:numerics-mt}--\ref{subsec:numerics-angle} do not by themselves establish uniform bounds for $C_{\mathrm{MT},h}$ and $\Lambda_{\kappa,h}$ under further refinement.

\begin{table}[htbp]
	\centering
	\caption{Angle defect diagnostics.}
	\label{tab:numerics-angle}
	\small
	\setlength{\tabcolsep}{9pt}
	\begin{tabular}{lrrcr}
		\toprule
		\multicolumn{1}{c}{Surface} & \multicolumn{1}{c}{DoFs}
		& \multicolumn{1}{c}{$\Lambda_{\kappa,h}$}
		& \multicolumn{1}{c}{Perturbation range of $\Lambda_{\kappa,h}$}
		& \multicolumn{1}{c}{$K_{\kappa,h}$} \\
		\midrule
		Sphere & 1572 & 0.778 & [0.730, 0.789] & $1.064\times10^{-1}$ \\
		& 6632 & 0.767 & [0.744, 0.779] & $2.845\times10^{-2}$ \\
		& 26432 & 0.801 & [0.771, 0.802] & $7.646\times10^{-3}$ \\
		& 105742 & 0.785 & [0.776, 0.805] & $2.076\times10^{-3}$ \\
		& 420952 & 0.804 & [0.798, 0.800] & $5.413\times10^{-4}$ \\
		\midrule
		Torus & 1970 & 3.380 & [3.181, 3.362] & $7.644\times10^{-1}$ \\
		& 7710 & 3.290 & [3.225, 3.531] & $1.667\times10^{-1}$ \\
		& 31070 & 3.856 & [3.673, 3.963] & $5.823\times10^{-2}$ \\
		& 124870 & 3.700 & [3.660, 3.697] & $1.437\times10^{-2}$ \\
		& 496630 & 3.642 & [3.611, 3.797] & $4.241\times10^{-3}$ \\
		\midrule
		Implicit surface & 2972 & 8.100 & [7.831, 8.491] & $7.749\times10^{-1}$ \\
		& 11982 & 14.342 & [12.967, 13.409] & $2.351\times10^{-1}$ \\
		& 46172 & 17.191 & [16.763, 17.376] & $6.693\times10^{-2}$ \\
		& 183392 & 18.061 & [17.539, 18.232] & $1.967\times10^{-2}$ \\
		& 746002 & 20.035 & [18.881, 19.445] & $4.967\times10^{-3}$ \\
		\bottomrule
	\end{tabular}
\end{table}

\begin{figure}[htbp]
	\centering
	\includegraphics[width=0.88\textwidth,
	alt={Angle-defect diagnostics on five mesh sizes per surface with a shared
		legend. Panel (a) shows normalized defects near 0.8 on the sphere, between
		3 and 4 on the torus, and rising from about 8 to 20 on the implicit surface.
		Panel (b) shows decreasing maximum absolute defects on logarithmic axes.
		Dashed circles and lines show original-mesh values; bars show ranges over
		the three perturbed meshes at the original mesh size.}]
	{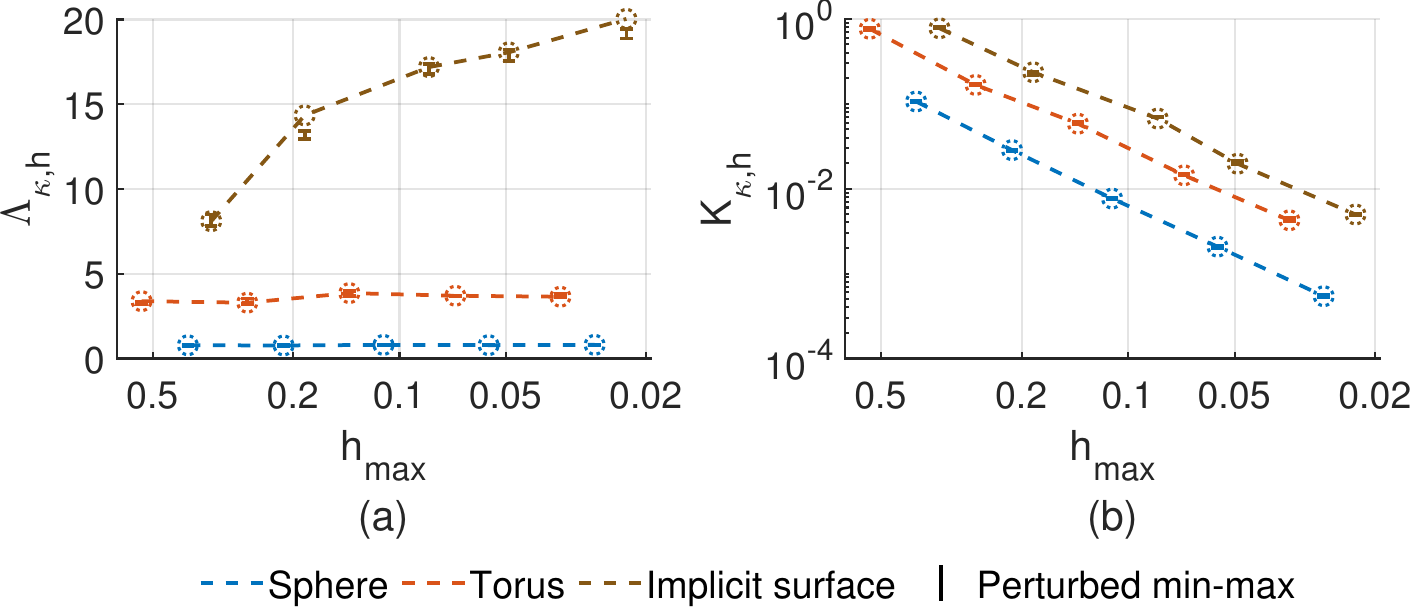}
	\caption{Angle defect scaling: (a) $\Lambda_{\kappa,h}$; (b) $K_{\kappa,h}$.}
	\label{fig:numerics-angle}
\end{figure}

\FloatBarrier

\section{Conclusions and Outlook}
\label{sec:conclusion}
The 12-dimensional REQ element is determined by vertex values and scalar edge moments and permits standard assembly without transporting vertex gradients. The main analytical difficulty is controlling the residuals left by these uncoupled gradients. Bounding the planar residual and the angle defect correction yields a weak discrete Miranda--Talenti-type estimate and uniform broken $H^2$ stability for the unstabilized Laplacian scheme on sufficiently refined interpolating surface meshes. Together with approximation and consistency estimates, this stability gives first-order convergence in the energy norm and second-order convergence in the broken $H^1$ and $L^2$ norms. The global discrete harmonic kernel consists only of constants on every connected closed mesh with non-degenerate triangles.

The three numerical examples exhibit these convergence rates. The additional tests show little variation in the Miranda--Talenti constants under the tested perturbations. The absolute angle defects decrease under refinement, although the normalized defects on the implicit surface have not reached a plateau. Further work includes extending the residual analysis to locally refined meshes and surfaces with boundary or evolving geometry, and developing effective preconditioners for the resulting systems.

\appendix

\section{Algebraic Verification of the Local Harmonic Kernel}
\label{app:local-kernel-algebra}

\subsection{Elimination of the enriched bubble function coefficients}
\label{app:local-kernel-bubbles}

We provide the two algebraic steps used in Lemma~\ref{lem:local-harmonic-kernel-new}. Write $v=p+\sum_{i=1}^3\vartheta_i b_i$, where $p\in P_4(K)$, and employ $(\lambda_2,\lambda_3)$ as coordinates. Write
\[
\Delta_K
=g_{22}\partial_{\lambda_2\lambda_2}
+2g_{23}\partial_{\lambda_2\lambda_3}
+g_{33}\partial_{\lambda_3\lambda_3},
\]
where $g_{22}=|\nabla\lambda_2|^2>0$, $g_{23}=\nabla\lambda_2\cdot\nabla\lambda_3$, and $g_{33}=|\nabla\lambda_3|^2>0$, satisfying $g_{22}g_{33}-g_{23}^2>0$.
Since $\Delta_Kp\in P_2(K)$, the cubic and quartic parts of $\Delta_K\sum_i\vartheta_i b_i$ vanish. Restricting the domain successively to $\lambda_3=0$, $\lambda_2=0$, and $(\lambda_2,\lambda_3)=(1-s,s)$ yields
\[
M(g_{22},g_{23},g_{33})
(\vartheta_1,\vartheta_2,\vartheta_3)^T=0,
\]
where
\[
M=
\begin{pmatrix}
	20g_{23}+8g_{33}&g_{33}&-3g_{33}\\
	10g_{23}+5g_{33}&g_{33}&-g_{33}\\
	3g_{22}&-g_{22}&-8g_{22}-20g_{23}\\
	g_{22}&-g_{22}&-5g_{22}-10g_{23}\\
	g_{22}+2g_{23}+g_{33}&-5g_{22}+5g_{33}&-g_{22}-2g_{23}-g_{33}\\
	-3g_{22}-6g_{23}-3g_{33}&8g_{22}-4g_{23}-12g_{33}&g_{22}+2g_{23}+g_{33}
\end{pmatrix}.
\]
The minors formed by rows $1,2,3$ and $1,2,4$, respectively, are given by
\[
\begin{aligned}
	&-g_{33}(70g_{22}g_{23}+11g_{22}g_{33}
	+200g_{23}^2+60g_{23}g_{33}),\\
	&-2g_{33}(20g_{22}g_{23}+3g_{22}g_{33}
	+50g_{23}^2+15g_{23}g_{33}).
\end{aligned}
\]
If both minors vanish, let $r_g=g_{22}/g_{33}$ and $s_g=g_{23}/g_{33}$. Elimination gives
\[
s_g(100s_g^2+40s_g+3)=0,
\qquad
r_g=-50s_g^2-15s_g.
\]
The candidate solutions are $(0,0)$, $(1,-1/10)$, and $(0,-3/10)$; positive definiteness leaves only $(1,-1/10)$. At this point, the minor formed by rows $1, 2,$ and $6$ equals $-14400g_{23}^3\neq0$. Therefore, $M$ has full column rank on every non-degenerate triangle, implying that $\vartheta_1=\vartheta_2=\vartheta_3=0$.

\subsection{Elimination of cubic and quartic harmonic parts}
\label{app:local-kernel-high-order}

Take a vertex of $K$ as the origin, identify the plane of $K$ with the complex plane, and denote complex coordinates by $\zeta$. Let $v_3$ and $v_4$ be the homogeneous cubic and quartic parts of $v$, respectively. Since $\Delta_Kv=0$, there exist $\alpha,\beta\in\mathbb C$ such that
\begin{equation}
	v_3(\zeta)=\operatorname{Re}(\alpha\zeta^3),
	\qquad
	v_4(\zeta)=\operatorname{Re}(\beta\zeta^4).
	\label{eq:homogeneous-harmonic-parts-app}
\end{equation}
Parts of degree at most two satisfy the reduction constraints in \eqref{eq:filtered-local-space} and can thus be subtracted. For an edge emanating from this vertex with terminal vector $\tau$, define $g_e(w)=w(0)-w(\tau)-12D_e^1(w)$.
Along this edge, write $\zeta=s\tau$. Direct integration using $v_k(s\tau)=s^kv_k(\tau)$ and $\partial_nv_k(s\tau)=s^{k-1}\partial_nv_k(\tau)$ yields
\[
g_e(v_3)=-\frac1{10}v_3(\tau),
\qquad
g_e(v_4)=-\frac15v_4(\tau),
\]
and
\[
D_e^{n,2,K}(v_3)=\frac1{180}\partial_nv_3(\tau),
\qquad
D_e^{n,2,K}(v_4)=\frac1{120}\partial_nv_4(\tau).
\]
Consequently, the two reduction constraints give
\[
v_3(\tau)+2v_4(\tau)=0,
\qquad
2\partial_nv_3(\tau)+3\partial_nv_4(\tau)=0.
\]
Reversing $n$ leaves the homogeneous conormal constraints invariant.
Write $\tau=re^{\mathrm i\phi}$, so that we can choose $n=\mathrm i e^{\mathrm i\phi}$. Differentiating \eqref{eq:homogeneous-harmonic-parts-app} directly along $n$ and combining it with the above two constraints yields
\[
\operatorname{Re}\bigl(\alpha e^{3\mathrm i\phi}
+2r\beta e^{4\mathrm i\phi}\bigr)=0,
\qquad
\operatorname{Im}\bigl(\alpha e^{3\mathrm i\phi}
+2r\beta e^{4\mathrm i\phi}\bigr)=0.
\]
Hence, $\alpha+2\beta\tau=0$. Applying this relation to two distinct edges emanating from the same point gives $\beta=\alpha=0$.

\section{Technical Details for the Vertex-Residual Estimates}
\label{app:appendix-b}

\subsection{Complete cancellation of interior edge integrals}
\label{app:edge-cancellation}

Let $e=K^+\cap K^-$ be an interior edge, and let $v_h\in V_h(\mathcal M_h)$. After rigidly aligning the element planes around $e$ and employing the common edge orientation, we prove
\begin{equation}
	\int_e[v_h]\partial_{tt}A_{n,e}(v_h)\dd s_h=0.
	\label{eq:complete-edge-integral-cancellation-app}
\end{equation}
Write $v_h^\pm=p^\pm+w^\pm$, where $p^\pm\in P_4(K^\pm)$ and $w^\pm$ belong to the space spanned by the enriched bubble functions, and decompose $A_{n,e}=A_{n,e}^{(4)}+A_{n,e}^{\rm enr}$ accordingly. Since $A_{n,e}^{(4)}|_e\in P_3(e)$, its second tangential derivative lies in $P_1(e)$; hence the polynomial contribution vanishes directly due to $[v_h]\perp P_1(e)$.

The remaining task is to handle the enriched contribution. On either element, denote the common edge by $e_k=\{\lambda_k=0\}$ and parameterize it via $t\in[0,1]$. Each bubble function vanishes on $e_k$, so $[v_h]|_{e_k}\in P_4(e_k)$. Endpoint continuity and $[v_h]\perp P_1(e_k)$ imply that $[v_h]|_{e_k}=c_k\chi(t)$, where $\chi(t)=t(1-t)(5t^2-5t+1)$. Among the three bubble functions, two contain $\lambda_k^2$ and have zero conormal derivatives on $e_k$; on both sides, the conormal derivatives of the remaining bubble function are proportional to the same profile function $\psi(t)=t^2(1-t)^2(1-2t)$. Therefore, its integral contribution vanishes owing to
\[
\int_0^1\chi(t)\psi''(t)\dd t=0.
\]
This argument is cyclic with respect to $k=1,2,3$, so the enriched contribution vanishes on all three edges. Combining the polynomial and enriched parts proves \eqref{eq:complete-edge-integral-cancellation-app}.

\subsection{Uniformity of local vertex patch estimates}
\label{app:uniform-open-fan}

Fix the number of associated triangles $m$. Translate the central vertex to the origin, scale $h_z$ to one, and fix the first radial direction. If $p_r$ are the successive radial endpoints, denote $G=(p_1,\ldots,p_{m+1})$. Set $K_r(G)=\operatorname{conv}\{0,p_r,p_{r+1}\}$. Closed planar patches satisfy $p_{m+1}=p_1$, whereas unfolded surface fans need not be closed. Shape regularity and quasi-uniformity ensure that the closure $\mathcal{G}_m$ of all normalized admissible configurations is a compact set, and their triangles are uniformly non-degenerate. Normalization preserves all local constraints and moment conditions; it also multiplies $A_z^{\rm flat}$ and $Q_z^{\rm flat}$ by the same factor, so their ratio remains invariant.

For $G\in\mathcal{G}_m$, let $X(G)$ be the normalized open fan space $X_z$ introduced in the proof of Proposition~\ref{lem:surface-patch-residual-new}. Its components are indexed by the labeled triangles $K_r(G)$, even when their planar images overlap. No condition is imposed between the two boundary edges of the open fan; in particular, the closed planar patch space is a subspace of $X(G)$. Using fixed bases of the ambient spaces $S(K_r(G))$, embed each $X(G)$ into $\mathcal{E}_m=\R^{18m}$ equipped with the Euclidean norm. The constraint matrices, the matrix of $A_z^{\rm flat}$, and the matrix of the symmetric bilinear form $B_z^{\rm flat}$ associated with $Q_z^{\rm flat}$ depend continuously on $G$. In particular, the limit of a sequence of pairs $(G_j,x_j)$ satisfying $x_j\in X(G_j)$ remains admissible, and compactness yields $M_m>0$ such that
\begin{equation}
	|B_z^{\mathrm{flat}}(x,y)|
	\leq M_m\|x\|_{\mathcal{E}_m}\|y\|_{\mathcal{E}_m},
	\qquad G\in\mathcal{G}_m,\quad x,y\in\mathcal{E}_m.
	\label{eq:uniform-B-bound-app}
\end{equation}

Using a fixed basis of \eqref{eq:flat-open-fan-kernel}, let $J_G:\R^5\to\mathcal{E}_m$ map its coefficients to their element-wise restrictions. Then $N(G):=\{x\in X(G):A_z^{\mathrm{flat}}(x)=0\}=\operatorname{im}J_G$. Since $J_G$ is continuous and has full column rank, $J_G^TJ_G$ is uniformly positive definite, and the orthogonal kernel projection $P_G:=J_G(J_G^TJ_G)^{-1}J_G^T$ is continuous with respect to $G$. We assert that there exists some $c_m>0$ satisfying
\begin{equation}
	A_z^{\mathrm{flat}}(x)
	\geq c_m\|x\|_{\mathcal{E}_m}^2
	\quad
	\forall\,G\in\mathcal{G}_m,
	\quad
	\forall\,x\in X(G)\text{ with }P_Gx=0.
	\label{eq:uniform-open-fan-coercivity-app}
\end{equation}
Otherwise, there exist $G_j\in\mathcal{G}_m$ and $x_j\in X(G_j)$ satisfying $P_{G_j}x_j=0$, $\|x_j\|_{\mathcal{E}_m}=1$, and $A_z^{\mathrm{flat}}(x_j)\to0$. Along some subsequence, $G_j\to G^*$ and $x_j\to x^*$. Admissibility and continuity give
\[
A_z^{\mathrm{flat}}(x^{*})=0,
\qquad P_{G^{*}}x^{*}=0,
\qquad \|x^{*}\|_{\mathcal{E}_m}=1.
\]
The first two relations indicate that $x^*$ both lies in $N(G^*)$ and is orthogonal to it, contradicting the third. This proves \eqref{eq:uniform-open-fan-coercivity-app}.

For arbitrary $x\in X(G)$, set $w=(I-P_G)x$. Then $w\in X(G)$ and $P_Gw=0$. Since $P_Gx$ belongs to the harmonic kernel in \eqref{eq:flat-open-fan-kernel}, $A_z^{\mathrm{flat}}(x)=A_z^{\mathrm{flat}}(w)$. The cyclic determinant sum in \eqref{eq:planar-fan-forms} is unchanged by adding a common harmonic quadratic polynomial, so $B_z^{\mathrm{flat}}(x,x)=B_z^{\mathrm{flat}}(w,w)$. Combining \eqref{eq:uniform-B-bound-app} with \eqref{eq:uniform-open-fan-coercivity-app} yields
\[
|Q_z^{\mathrm{flat}}(x)|
=|B_z^{\mathrm{flat}}(w,w)|
\leq M_m\|w\|_{\mathcal{E}_m}^2
\leq \frac{M_m}{c_m}A_z^{\mathrm{flat}}(x).
\]
Under the uniform valence upper bound, only a finite number of $m$ can occur. Consequently, undoing the normalization proves \eqref{eq:flat-open-fan-bound} with a uniform constant.

\end{document}